\documentclass[12pt]{article}
\usepackage{amsbsy,amsfonts,amsmath,amssymb,amsthm,bbm,enumerate}
\usepackage{graphicx}

 \numberwithin{equation}{section}
 \theoremstyle{plain}

\newtheorem{theorem}{Theorem}[section]

\newtheorem{lemma}[theorem]{Lemma}
\newtheorem{proposition}[theorem]{Proposition}

\newcommand{\dc}{d_{\rm c}}

\newcommand{\dpst}{\displaystyle}

\newcommand{\lbeq}[1]{\label{eq:#1}}
\newcommand{\mC}{{\mathbb C}}

\newcommand{\mP}{{\mathbb P}}
\newcommand{\mR}{{\mathbb R}}
\newcommand{\mZ}{{\mathbb Z}}
\newcommand{\nn}{\nonumber}
\newcommand{\ovec}{\boldsymbol{o}}
\newcommand{\pc}{p_{\rm c}}

\newcommand{\Rd}{\mR^d}

\newcommand{\refeq}[1]{(\ref{eq:#1})}
\newcommand{\sss}{\scriptscriptstyle}

\newcommand{\xvec}{\boldsymbol{x}}
\newcommand{\Zd}{\mZ^d}
\newcommand{\Zp}{\mZ_+}

\newcommand{\wvec}  {\boldsymbol{w}}
\newcommand{\yvec}  {\boldsymbol{y}}
\newcommand{\zvec}  {\boldsymbol{z}}

\title{Critical behavior and the limit distribution\\
for long-range oriented percolation.~II:\\
Spatial correlation}

\author{
Lung-Chi~Chen\footnote{Department of Mathematics, Fu-Jen Catholic
University, Taiwan.  {\tt lcchen@math.fju.edu.tw}}\\
Akira~Sakai\footnote{ Creative Research Initiative ``Sousei", Hokkaido
University, Japan.  {\tt sakai@cris.hokudai.ac.jp}}
\date{April 27, 2008\footnote{Updated: August 11, 2008.  The original
publication is available at www.springerlink.com.}}
}

\begin{document}
\maketitle

\begin{abstract}
We prove that the Fourier transform of the properly-scaled normalized two-point
function for sufficiently spread-out long-range oriented percolation with index
$\alpha>0$ converges to $e^{-C|k|^{\alpha\wedge2}}$ for some $C\in(0,\infty)$
above the upper-critical dimension $\dc\equiv2(\alpha\wedge2)$.  This answers
the open question remained in the previous paper \cite{cs08}.  Moreover, we
show that the constant $C$ exhibits crossover at $\alpha=2$, which is a result
of interactions among occupied paths.  The proof is based on a new method of
estimating fractional moments for the spatial variable of the lace-expansion
coefficients.
\end{abstract}

\section{Introduction and the main result}
We consider oriented bond percolation on $\Zd\times\Zp$, where each
time-oriented bond $((x,n),(y,n+1))$ is occupied with probability $pD(y-x)$ and
vacant with probability $1-pD(y-x)$, independently of the other bonds.  Here,
$D$ is a $\Zd$-symmetric probability distribution on $\Zd$, hence the parameter
$p\in[0,\|D\|_\infty^{-1}]$ can be interpreted as the average number of
occupied bonds per vertex. We say that a vertex $(x,j)$ is connected to
$(y,n)$, and write $(x,j)\to(y,n)$, if either $(x,j)=(y,n)$ or there is a
time-oriented path of occupied bonds from $(x,j)$ to $(y,n)$. Let $\mP_p$ be
the probability distribution of the bond variables, and define the two-point
function as
\begin{align*}
\varphi_p(x,n)=\mP_p\big((o,0)\to(x,n)\big),
\end{align*}
and its Fourier transform as
\begin{align*}
Z_p(k;n)=\sum_{x\in\Zd}e^{ik\cdot x}\varphi_p(x,n)\qquad(k\in[-\pi,\pi]^d).
\end{align*}
Notice that $Z_p(0;n)\equiv\sum_{x\in\Zd}\varphi_p(x,n)$ is the expected number
of vertices at time $n$ connected from $(o,0)$.  It has been known (\cite{gh02}
and references therein) that there is a $\pc\ge1$ such that
\begin{align*}
\chi_p\equiv\sum_{n=0}^\infty Z_p(0;n)
 \begin{cases}
 <\infty&(p<\pc),\\
 =\infty\quad&(p\ge\pc).
 \end{cases}
\end{align*}

In the previous paper \cite{cs08} (often referred to as Part~I from
now on), we investigated critical behavior of long-range oriented
percolation defined by
\begin{align*}
D(x)=\frac{h(x/L)}{\sum_{y\in\Zd}h(y/L)},
\end{align*}
where $h$ is a probability density function on $\Rd$ satisfying
$h(x)\asymp|x|^{-d-\alpha}$ (i.e., $|x|^{d+\alpha}h(x)$ is bounded
away from zero and infinity) for large $x$.  Here, $\alpha>0$ is the
characteristic index, and $L\in[1,\infty)$ is the parameter that
serves the model to spread out.  For example,
$\|D\|_\infty=O(\lambda)$, where
\begin{align*}
\lambda=L^{-d}.
\end{align*}
See \cite[Section~1.1]{cs08} for the precise definition and other
properties of $D$.  Notice that the variance
$\sigma^2\equiv\sum_{x\in\Zd}|x|^2D(x)$ does not exist if
$\alpha\le2$.

Suppose that there is a positive finite constant $v_\alpha$
($=\frac{\sigma^2}{2d}$ if $\alpha>2$) such that the Fourier transform $\hat
D(k)\equiv\sum_{x\in\Zd}D(x)e^{ik\cdot x}$ obeys the asymptotics
\begin{align}\lbeq{valpha}
1-\hat D(k)\underset{|k|\to0}\sim
 \begin{cases}
 v_\alpha|k|^{\alpha\wedge2}&(\alpha\ne2),\\
 v_2|k|^2\log\frac1{|k|}\quad&(\alpha=2).
 \end{cases}
\end{align}
The assumption \refeq{valpha} with $v_\alpha=O(L^{\alpha\wedge2})$ indeed holds
if, e.g., $h(x)\sim c|x|^{-d-\alpha}$ as $|x|\to\infty$ for some constant $c$
(see \cite[Section~10.5]{ks07} for the 1-dimensional case).  Let
\begin{align}\lbeq{kn-def}
k_n=k\times
 \begin{cases}
 (v_\alpha n)^{-\frac1{\alpha\wedge2}}&(\alpha\ne2),\\
 (v_2n\log\sqrt{n})^{-\frac12}\quad&(\alpha=2),
 \end{cases}
\end{align}
so that
\begin{align}\lbeq{n(1-D)}
\lim_{n\to\infty}n\big(1-\hat D(k_n)\big)=|k|^{\alpha\wedge2}.
\end{align}
Among various results, we proved that, for $\alpha>0$,
$d>2(\alpha\wedge2)$, $L\gg1$, $p\in(0,\pc]$ and $k\in\Rd$, there
exists $c,c'=1+O(\lambda)$ such that the normalized two-point
function satisfies
\begin{align}\lbeq{Zlim}
e^{-c|k|^{\alpha\wedge2}}\le\liminf_{n\to\infty}\frac{Z_p(k_n;n)}{Z_p(0;n)}\le
 \limsup_{n\to\infty}\frac{Z_p(k_n;n)}{Z_p(0;n)}\le e^{-c'|k|^{\alpha\wedge2}}.
\end{align}
Here, $\dc\equiv2(\alpha\wedge2)$ is the upper-critical dimension of this
model. We do not expect that \refeq{Zlim} holds for $d<\dc$.  Compare this
result with the behavior of the two-point function for the branching random
walk on $\Zd$ whose mean number of offspring per parent is $p>0$:
\begin{align}\lbeq{BRW}
Z_p^{\sss\text{BRW}}(k;n)=p^n\hat D(k)^n,&&
\lim_{n\to\infty}\frac{Z_p^{\sss\text{BRW}}(k_n;n)}{Z_p^{\sss\text{BRW}}(0;n)}
 =e^{-|k|^{\alpha\wedge2}}.
\end{align}
The latter is an immediate consequence of the former and \refeq{n(1-D)}.  We
note that $e^{-|k|^\alpha}$ is the characteristic function of an
$\alpha$-stable random variable (see, e.g., \cite{st94}).

The proof in \cite{cs08} of \refeq{Zlim} is based on the lace expansion for the
two-point function.  To derive information of the sequence $Z_p(k;n)$ from its
sum (= the Fourier-Laplace transform of the two-point function) and prove
\refeq{Zlim}, we established optimal control over fractional moments for the
\emph{time} variable of the lace-expansion coefficients. However, due to the
long-range nature of our $D$, we were unable to optimally control fractional
moments for the \emph{spatial} variable of the expansion coefficients and
squeeze the bounds in \refeq{Zlim} to identify the limit.  We note that, by the
standard Taylor-expansion method, the limit has been shown to exist at $p=\pc$
if $\alpha>2$ \cite{hs02} and for every $p\in(0,\pc]$ if the model is
finite-range \cite{NY2}.  This standard method does not work for $\alpha<2$ in
the current setting.

In this paper, we develop a new method to estimate fractional moments for the
spatial variable of the expansion coefficients and achieve the following result
on the normalized two-point function:

\begin{theorem}\label{thm:main}
Let $\alpha>0$, $d>2(\alpha\wedge2)$, $L\gg1$ and $p\in(0,\pc]$.  There is a
$C=1+O(\lambda)$ such that, for any $k\in\Rd$,
\begin{align*}
\lim_{n\to\infty}\frac{Z_p(k_n;n)}{Z_p(0;n)}=e^{-C|k|^{\alpha\wedge2}},
\end{align*}
where $k_n$ is defined in \refeq{kn-def}.  Moreover,
\begin{align}\lbeq{Cformula}
C=\frac1{\dpst1+pm_p\sum_{(x,n)}n\,\pi_p(x,n)\,m_p^n}\times
 \begin{cases}
 \dpst1+\frac{pm_p}{\sigma^2}\sum_{(x,n)}|x|^2\pi_p(x,n)\,m_p^n\quad
  &(\alpha>2),\\
 1&(\alpha\le2),
 \end{cases}
\end{align}
where $m_p$ is the radius of convergence for $\sum_{n=0}^\infty Z_p(0;n)\,m^n$,
and $\pi_p(x,n)$ is the alternating sum of the lace-expansion coefficients.
The sums in \refeq{Cformula} are absolutely convergent.
\end{theorem}

See, e.g., \cite[Section~3.1]{cs08} for the precise definition of $\pi_p(x,n)$.

The most remarkable observation in the above theorem is that the constant $C$
exhibits crossover at $\alpha=2$.  This phenomenon is observable if $\pi_p$,
which is model-dependent and contains information about interactions of
occupied paths, is nonzero.  We recall that, for the branching random walk,
occupied paths are independent and $\pi_p\equiv0$, hence $C$ is always 1 as in
\refeq{BRW}.  Therefore, the crossover behavior in \refeq{Cformula} is a result
of interactions among occupied paths.

We should emphasize that our approach developed in this paper and Part~I is
widely applicable, not only to our long-range oriented percolation, but also to
various other (long-range/finite-range) statistical-mechanical models.  For
example, our methods also apply to show that a similar result to the above
limit theorem holds for long-range self-avoiding walk with the characteristic
index $\alpha>0$, studied in \cite{hhs?}.  Markus Heydenreich is working in
this direction \cite{h?}.  His work will be a generalization of the results in
\cite{c00,yk88}, where $D(x)$ is proportional to $|x|^{-2}$ if $x$ is on the
coordinate axes, otherwise $D(x)=0$.  Since the coordinate axes are
1-dimensional, we should interpret $\alpha$ for this particular model as 1.

As another nontrivial application of the fractional-moment method of this
paper, one of the authors (LCC) will report in his ongoing work that the
gyration radius $\xi_p^{\sss(r)}$ of order $r\in(0,\alpha)$ for sufficiently
spread-out oriented percolation with $d>2(\alpha\wedge2)$ obeys
\begin{align*}
\xi_p^{\sss(r)}(n)\equiv\bigg(\frac1{Z_p(0;n)}\sum_{x\in\Zd}|x|^r\varphi_p
 (x,n)\bigg)^{1/r}\asymp
 \begin{cases}
 n^{\frac1{\alpha\wedge2}}&(\alpha\ne2),\\
 (n\log n)^{1/2}\quad&(\alpha=2),
 \end{cases}
\end{align*}
for every $p\in(0,\pc]$.

The rest of the paper is organized as follows.  In Section~\ref{s:recap}, we
summarize the relevant results from Part~I.  In Section~\ref{s:proof-thm}, we
prove Theorem~\ref{thm:main} subject to a key proposition on fractional moments
for the spatial variable of the lace-expansion coefficients.  We prove that
proposition in Section~\ref{s:proof-prp} using a certain integral
representation for fractional powers of positive reals.

\section{Summary of the relevant results from Part~I}\label{s:recap}
In this section, we summarize the results from Part~I that will be used in the
rest of the paper.

First, we introduce some notation.  Let
\begin{align*}
q_p(x,n)=\mP_p\Big(\big((o,0),(x,n)\big)\text{ is occupied}\Big)
 \equiv
 \begin{cases}
 pD(x)\quad&(n=1),\\
 0&(n\ne1).
 \end{cases}
\end{align*}
We denote the space-time convolution of functions $f$ and $g$ on
$\Zd\times\Zp$ by
\begin{align*}
(f*g)(x,n)=\sum_{(y,t)\in\Zd\times\Zp}f(y,t)\;g(x-y,n-t),
\end{align*}
and the Fourier-Laplace transform of $f$ by
\begin{align*}
\hat f(k,z)=\sum_{(x,n)\in\Zd\times\Zp}f(x,n)\,e^{ik\cdot x}z^n\qquad(k\in
 [-\pi,\pi]^d,~z\in\mC).
\end{align*}
Notice that $\frac{-1}2\Delta_k\hat f(l,z)$, defined as
\begin{align}\lbeq{DeltaFL}
\frac{-1}2\Delta_k\hat f(l,z)&=\hat f(l,z)-\frac{\hat f(l+k,z)+\hat f(l-k,z)}
 2\nn\\
&\equiv\sum_{(x,n)\in\Zd\times\Zp}\big(1-\cos(k\cdot x)\big)f(x,n)\,e^{il\cdot
 x}z^n,
\end{align}
is the Fourier-Laplace transform of $(1-\cos(k\cdot x))f(x,n)$.

In \cite[Section~3.1]{cs08}, we explained the derivation of the convolution
equation
\begin{align*}
\varphi_p(x,n)=\pi_p(x,n)+(\pi_p*q_p*\varphi_p)(x,n),
\end{align*}
where $\pi_p(x,n)$ is the alternating sum of the $\Zd$-symmetric nonnegative
lace-expansion coefficients $\pi_p^{\sss(N)}(x,n)$ for $N=0,1,2,\dots$:
\begin{align}\lbeq{piN-exp}
\pi_p(x,n)=\sum_{N=0}^\infty(-1)^N\pi_p^{\sss(N)}(x,n).
\end{align}
The precise definition of $\pi_p^{\sss(N)}$ is unimportant in this paper.
However, we will use the following properties of $\pi_p$ and $\varphi_p$:

\begin{proposition}
Let $\alpha>0$, $d>2(\alpha\wedge2)$ and $L\gg1$.  Then,
\begin{align}
pm_p\hat\pi_p(0,m_p)&=1,\lbeq{pmppi=1}\\
\sum_{(x,n)}n\,|\pi_p(x,n)|\,m^n&\le O(\lambda),\lbeq{npi-bd}\\
\sum_{(x,n)}\big(1-\cos(k\cdot x)\big)|\pi_p(x,n)|\,m^n&\le O(\lambda)\big(1
 -\hat D(k)\big),\lbeq{cospi-bd}
\end{align}
and
\begin{align}
|\hat\varphi_p(k,me^{i\theta})|&\le\frac{O(1)}{pm_p(1-\frac{m}{m_p})+|\theta|
 +1-\hat D(k)},\lbeq{varphi-bd}\\
|\Delta_k\hat\varphi_p(l,me^{i\theta})|&\le\sum_{(j,j')=(0,\pm1),(1,-1)}
 \frac{1-\hat D(k)}{pm_p(1-\frac{m}{m_p})+|\theta|+1-\hat D(l+jk)}\nn\\
&\hskip6pc\times\frac{O(1)}{pm_p(1-\frac{m}{m_p})+|\theta|+1-\hat D(l+j'k)},
 \lbeq{Deltavarphi-bd}
\end{align}
uniformly in $p\in(0,\pc]$, $m\in[0,m_p)$, $k,l\in[-\pi,\pi]^d$ and
$\theta\in[-\pi,\pi]$.
\end{proposition}

\begin{proof}
The identity \refeq{pmppi=1} for every $p\in(0,\pc]$ was proved in
\cite[(2.17) and (2.22)]{cs08}.  The bounds \refeq{npi-bd} and
\refeq{varphi-bd} for $p\in(0,\pc)$ were also proved in Part~I, and
can be extended up to $p=\pc$, as long as $m$ is strictly less than
the radius of convergence, $m_{\pc}=1$ (cf.,
\cite[Corollary~1.3]{cs08}).

The same extension applies to the bounds \refeq{cospi-bd} and
\refeq{Deltavarphi-bd}, if they hold uniformly in $p\in(0,\pc)$ and
$m\in[0,m_p)$.  In Part~I, we showed that
\begin{align*}
\sum_{(x,n)}\big(1-\cos(k\cdot x)\big)|\pi_p(x,n)|\,m^n\le O(\lambda)\bigg(1-
 \frac{m}{m_p}+1-\hat D(k)\bigg),
\end{align*}
uniformly in $p\in(0,\pc)$, $m\in[0,m_p)$ and $k\in[-\pi,\pi]^d$.
However, since the left-hand side is increasing in $m<m_p$, we
obtain
\begin{align*}
\sum_{(x,n)}\big(1-\cos(k\cdot x)\big)|\pi_p(x,n)|\,m^n&\le\lim_{m\uparrow m_p}
 \sum_{(x,n)}\big(1-\cos(k\cdot x)\big)|\pi_p(x,n)|\,m^n\\
&\le O(\lambda)\lim_{m\uparrow m_p}\bigg(1-\frac{m}{m_p}+1-\hat D(k)\bigg)\\
&\le O(\lambda)\big(1-\hat D(k)\big),
\end{align*}
as required.  Using this stronger bound and following the steps in
\cite[Section~4.2]{cs08}, we also obtain \refeq{Deltavarphi-bd}.
This completes the proof.
\end{proof}

Finally, we summarize the results for the $n^\text{th}$ coefficient $Z_p(k;n)$
of the series expansion of $\hat\varphi_p(k,m)$ in powers of $m$:
$\hat\varphi_p(k,m)\equiv\sum_{n=0}^\infty Z_p(k;n)\,m^n$.
Let (cf., \cite[(2.33)--(2.34)]{cs08})
\begin{align}
\hat A_p^{\sss(1)}(k)&=\hat D(k)+\frac{m_p\partial_m\hat\pi_p(k,m_p)}{pm_p
 \hat\pi_p(k,m_p)^2},\lbeq{A1-def}\\
\hat B_p(k)&=1-\hat D(k)+\frac{\hat\pi_p(0,m_p)-\hat\pi_p(k,m_p)}{\hat\pi_p
 (k,m_p)},\lbeq{B-def}
\end{align}
where $\partial_m\hat\pi_p(k,m_p)=\lim_{m\uparrow m_p}\partial_m\hat\pi_p (k,m)$.
Notice that, by \refeq{npi-bd}--\refeq{cospi-bd},
\begin{align}
|m_p\partial_m\hat\pi_p(k,m_p)|&\le\lim_{m\uparrow m_p}\sum_{(x,n)}n\,|\pi_p
 (x,n)|\,m^n\le O(\lambda),\lbeq{partialpi-bd}\\
|\hat\pi_p(0,m_p)-\hat\pi_p(k,m_p)|&\le\sum_{(x,n)}\big(1-\cos(k\cdot x)\big)
 |\pi_p(x,n)|\,m_p^n\le O(\lambda)\big(1-\hat D(k)\big),\lbeq{Deltapi-bd}
\end{align}
where the $O(\lambda)$ terms are uniform in $p\in(0,\pc]$ and
$k\in[-\pi,\pi]^d$.  Moreover, since $\pi_p(x,0)$ equals the Kronecker delta
$\delta_{x,o}$ (cf., \cite[(3.2)]{cs08}), we have
$\hat\pi_p(k,m_p)=1+O(\lambda)$ and thus $\hat A_p^{\sss(1)}(k)+\hat
B_p(k)=1+O(\lambda)$ uniformly in $p\in(0,\pc]$ and $k\in[-\pi,\pi]^d$.

In \cite[Section~2.4]{cs08}, we showed that, for $\alpha>0$,
$d>2(\alpha\wedge2)$, $L\gg1$ and
$\epsilon\in(0,1\wedge\frac{d-2(\alpha\wedge2)}{\alpha\wedge2})$,
\begin{align*}
m_p^nZ_p(k;n)=\frac1{pm_p\big(\hat A_p^{\sss(1)}(k)+\hat B_p(k)\big)}\bigg(
 \frac{\hat A_p^{\sss(1)}(k)}{\hat A_p^{\sss(1)}(k)+\hat B_p(k)}\bigg)^n
 +O(n^{-\epsilon}),
\end{align*}
hence
\begin{align}\lbeq{Zk/Z0}
\frac{Z_p(k;n)}{Z_p(0;n)}=\frac{\hat A_p^{\sss(1)}(0)}{\hat A_p^{\sss(1)}(k)
 +\hat B_p(k)}\bigg(\frac{\hat A_p^{\sss(1)}(k)}{\hat A_p^{\sss(1)}(k)+\hat
 B_p(k)}\bigg)^n+O(n^{-\epsilon}),
\end{align}
uniformly in $p\in(0,\pc]$ and $k\in[-\pi,\pi]^d$.  To prove
Theorem~\ref{thm:main}, it thus suffices to investigate the first term in
\refeq{Zk/Z0}.

\section{Proof of Theorem~\ref{thm:main} subject to a key proposition}
 \label{s:proof-thm}
In this section, we first prove Theorem~\ref{thm:main} assuming convergence of
fractional moments for the spatial variable of $\pi_p$, as stated in the
following proposition:

\begin{proposition}\label{prp:key}
Let $\alpha>0$, $d>2(\alpha\wedge2)$, $L\gg1$ and
\begin{align}\lbeq{delta-def}
\delta
 \begin{cases}
 \in\big(0,\alpha\wedge2\wedge(d-2(\alpha\wedge2))\big)\quad&(\alpha\ne2),\\
 =0&(\alpha=2).
 \end{cases}
\end{align}
Then, for any $p\in(0,\pc]$,
\begin{align}\lbeq{|x|pi}
\sum_{(x,n)}|x|^{\alpha\wedge2+\delta}|\pi_p(x,n)|\,m_p^n<\infty.
\end{align}
\end{proposition}

We will roughly explain why $\delta$ is chosen as in \refeq{delta-def},
after the proof of Theorem~\ref{thm:main} is completed.  The proof of
Proposition~\ref{prp:key} is deferred to Section~\ref{s:proof-prp}.

\begin{proof}[Proof of Theorem~\ref{thm:main} subject to Proposition~\ref{prp:key}]
As explained at the end of Section~\ref{s:recap}, it suffices to investigate
the term
\begin{align*}
\bigg(\frac{\hat A_p^{\sss(1)}(k)}{\hat A_p^{\sss(1)}(k)+\hat B_p(k)}\bigg)^n
 \equiv\Bigg(\bigg(1+\frac{\hat B_p(k)}{\hat A_p^{\sss(1)}(k)}\bigg)^{-\frac{
 \hat A_p^{\sss(1)}(k)}{\hat B_p(k)}}\Bigg)^{\frac{n(1-\hat D(k))}{\hat A_p^{
 \sss(1)}(k)}\frac{\hat B_p(k)}{1-\hat D(k)}}.
\end{align*}
Notice that, by \refeq{A1-def}--\refeq{Deltapi-bd} and \refeq{n(1-D)},
\begin{align*}
\bigg(1+\frac{\hat B_p(k)}{\hat A_p^{\sss(1)}(k)}\bigg)^{-\frac{\hat A_p^{\sss
 (1)}(k)}{\hat B_p(k)}}\underset{|k|\to0}\to e^{-1},&&&&
\frac{n(1-\hat D(k_n))}{\hat A_p^{\sss(1)}(k_n)}\underset{n\to\infty}\to
 \frac{|k|^{\alpha\wedge2}}{\hat A_p^{\sss(1)}(0)},
\end{align*}
where
\begin{align*}
\hat A_p^{\sss(1)}(0)&=1+pm_p\sum_{(x,n)}n\,\pi_p(x,n)\,m_p^n.
\end{align*}
Moreover,
\begin{align*}
\frac{\hat B_p(k)}{1-\hat D(k)}&=1+pm_p\,\frac{\hat\pi_p(0,m_p)}{\hat\pi_p(k,
 m_p)}\,\frac{\hat\pi_p(0,m_p)-\hat\pi_p(k,m_p)}{1-\hat D(k)}\nn\\
&\underset{|k|\to0}\to1+pm_p\lim_{|k|\to0}\frac{\hat\pi_p(0,m_p)
 -\hat\pi_p(k,m_p)}{1-\hat D(k)},
\end{align*}
if the limit exists.  To complete the proof of Theorem~\ref{thm:main}, it
remains to show
\begin{align}\lbeq{pi/1-D}
\lim_{|k|\to0}\frac{\hat\pi_p(0,m_p)-\hat\pi_p(k,m_p)}{1-\hat D(k)}=
 \begin{cases}
 \dpst\frac1{\sigma^2}\sum_{(x,n)}|x|^2\pi_p(x,n)\,m_p^n\quad&(\alpha>2),\\
 0&(\alpha\le2).
 \end{cases}
\end{align}

Now we choose $\delta$ as in \refeq{delta-def} and use
Proposition~\ref{prp:key} to prove \refeq{pi/1-D} for (i)~$\alpha\le2$ and
(ii)~$\alpha>2$, separately.

\smallskip

(i) Let $\alpha\le2$ and $\alpha+\delta\le2$.  Then, we have
\begin{align*}
0\le1-\cos(k\cdot x)\le O(|k\cdot x|^{\alpha+\delta}).
\end{align*}
By the spatial symmetry of the model and using \refeq{|x|pi} with $\delta$
satisfying $\alpha+\delta\le2$ and \refeq{delta-def},
\begin{align*}
|\hat\pi_p(0,m_p)-\hat\pi_p(k,m_p)|&=\bigg|\sum_{(x,n)\in\Zd\times\Zp}
 \big(1-\cos(k\cdot x)\big)\,\pi_p(x,n)\,m_p^n\bigg|\\
&\le O(|k|^{\alpha+\delta})\sum_{(x,n)\in\Zd\times\Zp}|x|^{\alpha+\delta}
 |\pi_p(x,n)|\,m_p^n\\
&=O(|k|^{\alpha+\delta}).
\end{align*}
By \refeq{valpha}, we thus obtain that, for small $|k|$,
\begin{align*}
\frac{|\hat\pi_p(0,m_p)-\hat\pi_p(k,m_p)|}{1-\hat D(k)}\le
 \begin{cases}
 O(|k|^\delta)&(\alpha<2),\\
 O(1/\log\frac1{|k|})\quad&(\alpha=2).
 \end{cases}
\end{align*}
This yields \refeq{pi/1-D} for $\alpha\le2$.

\smallskip

(ii) Let $\alpha>2$ and $\delta\le2$.  By the Taylor expansion,
\begin{align*}
1-\cos(k\cdot x)=\frac{(k\cdot x)^2}2+O(|k\cdot x|^{2+\delta}).
\end{align*}
Then, by the spatial symmetry of the model and using \refeq{|x|pi} with $\delta$ satisfying \refeq{delta-def},
\begin{align}\lbeq{pi0-pik>2}
\hat\pi_p(0,m_p)-\hat\pi_p(k,m_p)=\frac{|k|^2}{2d}\sum_{(x,n)\in\Zd\times\Zp}
 |x|^2\pi_p(x,n)\,m_p^n+O(|k|^{2+\delta}).
\end{align}
The limit \refeq{pi/1-D} for $\alpha>2$ follows from \refeq{pi0-pik>2} and the
asymptotics \refeq{valpha} with $v_\alpha=\frac{\sigma^2}{2d}$.

\smallskip

This completes the proof of Theorem~\ref{thm:main} subject to
Proposition~\ref{prp:key}.
\end{proof}

Before closing this section, we roughly explain why
$\delta<\alpha\wedge2\wedge(d-2(\alpha\wedge2))$ for $\alpha\ne2$ (the
necessity of $\delta=0$ for $\alpha=2$ and $\delta>0$ for $\alpha\ne2$ is
obvious from the above proof of Theorem~\ref{thm:main}).  This is a sort of
preview of Section~\ref{s:proof-prp}.

In Section~\ref{s:proof-prp}, we will use diagrammatic bounds on the expansion
coefficients $\pi_p^{\sss(N)}$ in \refeq{piN-exp}.  In each bound (cf.,
\refeq{pi0-dec}--\refeq{piN-dec2} below), there are \emph{two} sequences of
two-point functions from $(o,0)$ to $(x,n)$.  To bound
$\sum_{(x,n)}|x|^{\alpha\wedge2+\delta}\pi_p^{\sss(N)}(x,n)m^n$, we will split
the power $\alpha\wedge2+\delta$ into $\delta_1$ and $\delta_2$, and multiply
one of the aforementioned two sequences of two-point functions by
$|x|^{\delta_1}$ and the other by $|x|^{\delta_2}$.  Here, we choose $\delta_1$
and $\delta_2$ both less than $\alpha\wedge2$, so as to potentially control the
weighted two-point functions, like $|y|^{\delta_1}\varphi_p(y,s)$.

Then, $\sum_{(x,n)}|x|^{\alpha\wedge2+\delta}\pi_p^{\sss(N)}(x,n)m^n$ will be
bounded by the product of diagram functions (cf., Lemma~\ref{lmm:|x1|piN-bd}
below).  Those diagram functions are the ``triangle" $T_{p,m}$, which is
independent of $\delta_1$ and $\delta_2$, its weighted version
$T'_{p,m}(\delta_1)$ and the weighted ``bubbles" $W'_{p,m}(\delta_2)$ and
$W''_{p,m}(\delta_1,\delta_2)$ (cf., \refeq{T-def}--\refeq{W''-def} below).  As
shown in Section~\ref{ss:bounds}, it is not hard to bound $W'_{p,m}(\delta_2)$
uniformly in $p$ and $m$ for $d>2(\alpha\wedge2)$ and $L\gg1$ as long as
$\delta_2<\alpha\wedge2$.  However, to bound $T'_{p,m}(\delta_1)$ and
$W''_{p,m}(\delta_1,\delta_2)$ uniformly in $p$ and $m$, we will have to choose
$\delta_1$ to be small depending on how close $d$ is to the upper-critical
dimension $2(\alpha\wedge2)$.  As described in Lemma~\ref{lmm:Ii-bds} below, we
will choose $\delta_1$ less than $d-2(\alpha\wedge2)$.

To summarize the above, we have
\begin{align*}
0<\delta_1<\alpha\wedge2\wedge\big(d-2(\alpha\wedge2)\big),&&
0<\delta_2<\alpha\wedge2,&&
\delta_1+\delta_2=\alpha\wedge2+\delta.
\end{align*}
To satisfy all, it suffices to choose $\delta_1$ ``slightly" larger than
$\delta$ and let $\delta_2=\alpha\wedge2-(\delta_1-\delta)$.  This is why we
choose $\delta<\alpha\wedge2\wedge(d-2(\alpha\wedge2))$ when $\alpha\ne2$.

\section{Proof of Proposition~\ref{prp:key}}\label{s:proof-prp}
Finally, in this section, we prove Proposition~\ref{prp:key}.  First, in
Section~\ref{ss:diagbds}, we bound fractional moments for the spatial variable
of the expansion coefficients $\pi_p^{\sss(N)}$ in \refeq{piN-exp} in terms of
certain diagram functions.  In Section~\ref{ss:prelim}, we use an integral
representation of $a^\delta$ for $a>0$ and $\delta\in(0,2)$, which is the key to the proof of Proposition~\ref{prp:key}. In Section~\ref{ss:bounds}, we show that the aforementioned diagram functions are convergent, and complete the proof of
Proposition~\ref{prp:key}.

\subsection{Diagrammatic bounds on the expansion coefficients}
 \label{ss:diagbds}
In this subsection, we bound $\sum_{(x,n)}|x|^r|\pi_p(x,n)|\,m^n$ for $r>0$ in
terms of the diagram functions $T_{p,m}$, $T'_{p,m}$, $W'_{p,m}$ and
$W''_{p,m}$ defined in Lemma~\ref{lmm:|x1|piN-bd} below.

First, we show the following elementary inequality:

\begin{lemma}\label{lmm:elementary}
For any $r>0$ and $m\ge0$,
\begin{align*}
\sum_{(x,n)}|x|^r|\pi_p(x,n)|\,m^n\le d^{\frac{r}2+1}\sum_{N=0}^\infty
 \sum_{(x,n)}|x_1|^r\pi_p^{\sss(N)}(x,n)\,m^n,
\end{align*}
where $x_1$ is the first coordinate of $x\equiv(x_1,\dots,x_d)$.
\end{lemma}

\begin{proof}
For any $r>0$, we have
\begin{align*}
|x|^r=\bigg(\sum_{j=1}^d|x_j|^2\bigg)^{r/2}\le\bigg(\sum_{j=1}^d\|x\|_r^2
 \bigg)^{r/2}=d^{r/2}\|x\|_r^r\equiv d^{r/2}\sum_{j=1}^d|x_j|^r.
\end{align*}
By this inequality and using the nonnegativity and the spatial symmetry of
$\pi_p^{\sss(N)}$, we obtain
\begin{align*}
\sum_{(x,n)}|x|^r|\pi_p(x,n)|\,m^n&\le d^{r/2}\sum_{j=1}^d\sum_{(x,n)}|x_j|^r
 \bigg|\sum_{N=0}^\infty(-1)^N\pi_p^{\sss(N)}(x,n)\bigg|\,m^n\\
&\le d^{r/2}\sum_{j=1}^d\sum_{N=0}^\infty\sum_{(x,n)}|x_j|^r\pi_p^{\sss(N)}(x,
 n)\,m^n\\
&=d^{\frac{r}2+1}\sum_{N=0}^\infty\sum_{(x,n)}|x_1|^r\pi_p^{\sss(N)}(x,n)\,m^n,
\end{align*}
as required.
\end{proof}

Next, we use \cite[Lemma~1]{s??} to investigate
$\sum_{(x,n)}|x_1|^r\pi_p^{\sss(N)}(x,n)\,m^n$.  For notational convenience,
we denote vertices in $\mZ^{d+1}$ by bold letters, e.g., $\ovec\equiv(o,0)$ and
$\xvec=(x,t_{\xvec})$, where $t_{\xvec}$ is the temporal part of $\xvec$.  Let
\begin{align*}
\psi_p(\xvec)=(q_p*\varphi_p)(\xvec).
\end{align*}
Given a sequence of vertices $\yvec_1,\dots,\yvec_j\in\mZ^{d+1}$, we write
\begin{align*}
\vec\yvec_j=\sum_{i=1}^j\yvec_i.
\end{align*}
For $\yvec_1,\zvec_1,\yvec_2,\zvec_2,\dots\in\mZ^{d+1}$, we define
\begin{align*}
\Lambda_p(\vec\yvec_{i-1},\vec\zvec_{i-1};\vec\yvec_i,\vec\zvec_i)&=\psi_p
 (\yvec_i)\,\psi_p(\zvec_i)\,\frac{\varphi_p(\vec\yvec_i-\vec\zvec_i)+
 \varphi_p(\vec\zvec_i-\vec\yvec_i)}{2^{\delta_{\vec\yvec_i,\vec\zvec_i}}},\\
\tilde\Lambda_p(\vec\yvec_i,\vec\zvec_i;\vec\yvec_{i+1},\vec\zvec_{i+1})&=
 \frac{\varphi_p(\vec\yvec_i-\vec\zvec_i)+\varphi_p(\vec\zvec_i-\vec\yvec_i)}
 {2^{\delta_{\vec\yvec_i,\vec\zvec_i}}}\,\psi_p(\yvec_{i+1})\,\psi_p(\zvec_{i
 +1}).
\end{align*}

\begin{lemma}[Equivalent to Lemma~1 \cite{s??}]\label{lmm:diag}
For $N=0$,
\begin{align}\lbeq{pi0-dec}
0\le\pi_p^{\sss(0)}(\xvec)-\delta_{\xvec,\ovec}\le\psi_p(\xvec)^2.
\end{align}
For $N\ge1$,
\begin{align}\lbeq{piN-dec1}
\pi_p^{\sss(N)}(\xvec)&\le\sum_{\substack{\yvec_1,\dots,\yvec_{N+1}\\ \zvec_1,
 \dots,\zvec_{N+1}\\ (\vec\yvec_{N+1}=\vec\zvec_{N+1}=\xvec)\\ (t_{\yvec_1}\ge
 t_{\zvec_1})}}\varphi_p(\yvec_1)\,\varphi_p(\zvec_1)\prod_{i=1}^N\tilde
 \Lambda_p(\vec\yvec_i,\vec\zvec_i;\vec\yvec_{i+1},\vec\zvec_{i+1}),
\end{align}
and, for any $j\in\{2,\dots,N+1\}$,
\begin{align}\lbeq{piN-dec2}
\pi_p^{\sss(N)}(\xvec)\le\sum_{\substack{\yvec_1,\dots,\yvec_{N+1}\\ \zvec_1,
 \dots,\zvec_{N+1}\\ (\vec\yvec_{N+1}=\vec\zvec_{N+1}=\xvec)}}&\varphi_p
 (\yvec_1)\,\varphi_p(\zvec_1)\,\varphi_p(\yvec_1-\zvec_1)\bigg(\prod_{i=2}^{j
 -1}\Lambda_p(\vec\yvec_{i-1},\vec\zvec_{i-1};\vec\yvec_i,\vec\zvec_i)\bigg)
 \nn\\
&\times\psi_p(\yvec_j)\,\psi_p(\zvec_j)\bigg(\prod_{i=j}^N\tilde\Lambda_p(\vec
 \yvec_i,\vec\zvec_i;\vec\yvec_{i+1},\vec\zvec_{i+1})\bigg),
\end{align}
where an empty product is regarded as 1.
\end{lemma}

For further notational convenience, we let
\begin{align*}
\varphi_p^{\sss(m)}(\xvec)&=\varphi_p(\xvec)\,m^{t_{\xvec}},\\
\psi_p^{\sss(m)}(\xvec)&=\psi_p(\xvec)\,m^{t_{\xvec}},\\
\Lambda_p^{\sss(m)}(\vec\yvec_{i-1},\vec\zvec_{i-1};\vec\yvec_i,\vec\zvec_i)
 &=\Lambda_p(\vec\yvec_{i-1},\vec\zvec_{i-1};\vec\yvec_i,\vec\zvec_i)\,m^{t_{
 \zvec_i}},\\
\tilde\Lambda_p^{\sss(m)}(\vec\yvec_i,\vec\zvec_i;\vec\yvec_{i+1},\vec\zvec_{i
 +1})&=\tilde\Lambda_p(\vec\yvec_i,\vec\zvec_i;\vec\yvec_{i+1},\vec\zvec_{i+1})
 \,m^{t_{\zvec_{i+1}}}.
\end{align*}
Given arbitrary $\delta_1,\delta_2>0$, we define $T_{p,m}$, $T'_{p,m}\equiv
T'_{p,m}(\delta_1)$, $W'_{p,m}\equiv W'_{p,m}(\delta_2)$ and $W''_{p,m}\equiv
W''_{p,m}(\delta_1,\delta_2)$ as
\begin{align}
T_{p,m}&=\sup_{\xvec}\sum_{\yvec}\psi_p(\yvec)\bigg((\psi_p^{\sss(m)}*\varphi_p
 )(\yvec-\xvec)+\sum_{\zvec}\varphi_p(\zvec-\yvec)\,\psi_p^{\sss(m)}(\zvec-\xvec)
 \bigg),\lbeq{T-def}\\
T'_{p,m}&=\sup_{\xvec}\sum_{\yvec}|y_1|^{\delta_1}\psi_p(\yvec)\bigg((\psi_p^{
 \sss(m)}*\varphi_p)(\yvec-\xvec)+\sum_{\zvec}\varphi_p(\zvec-\yvec)\,\psi_p^{\sss(m)}(\zvec-\xvec)
 \bigg),\lbeq{T'-def}\\
W'_{p,m}&=\sup_{\xvec}\sum_{\yvec}\psi_p(\yvec)\,|y_1-x_1|^{\delta_2}\psi_p^{
 \sss(m)}(\yvec-\xvec),\lbeq{W'-def}\\
W''_{p,m}&=\sup_{\xvec}\sum_{\yvec}|y_1|^{\delta_1}\psi_p(\yvec)\,|y_1-x_1|^{
 \delta_2}\psi_p^{\sss(m)}(\yvec-\xvec).\lbeq{W''-def}
\end{align}
\begin{figure}[t]
\begin{align*}
T_{p,m}&=~~\raisebox{-2pc}{\includegraphics[scale=0.15]{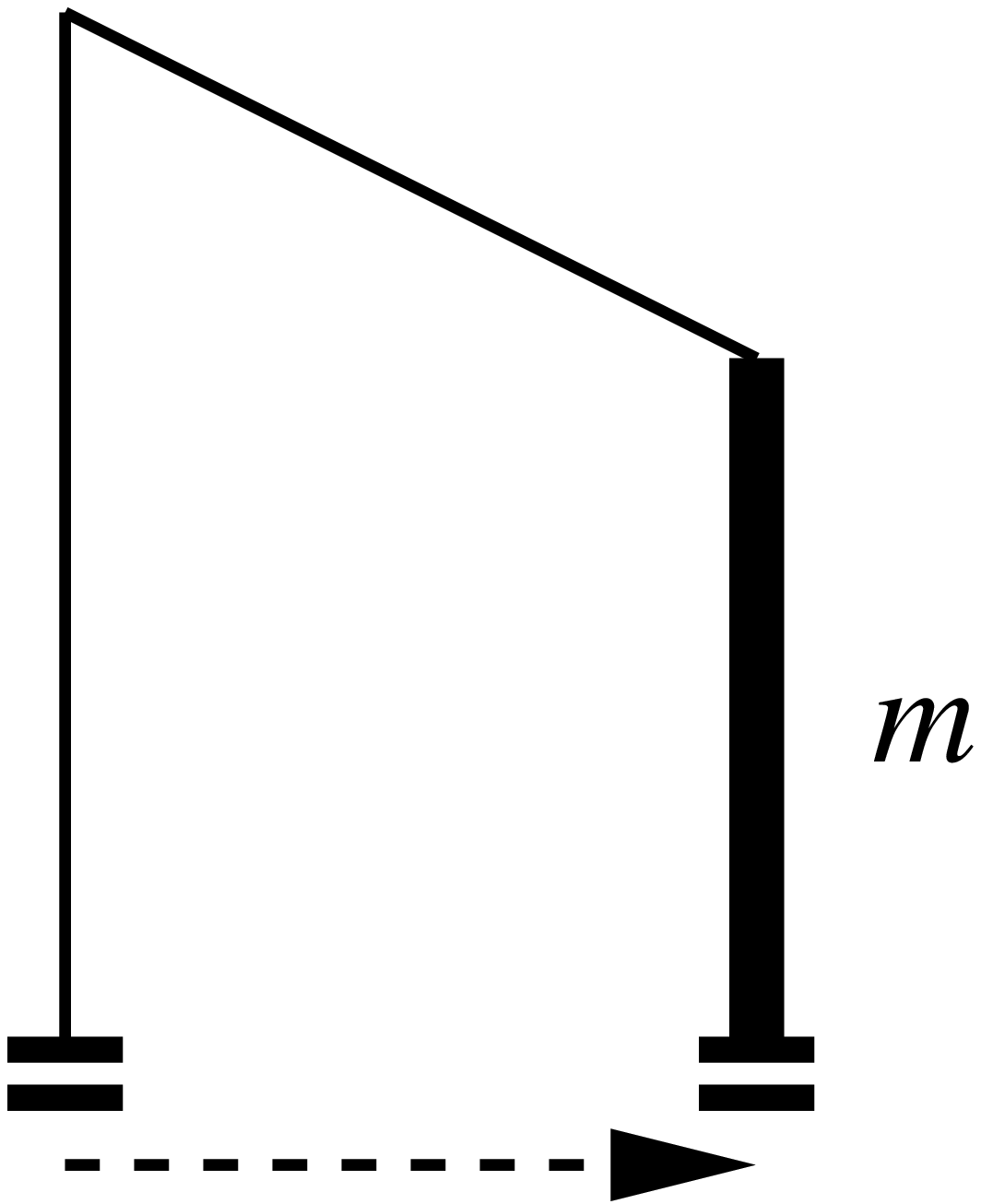}}~~+~~
 \raisebox{-2pc}{\includegraphics[scale=0.15]{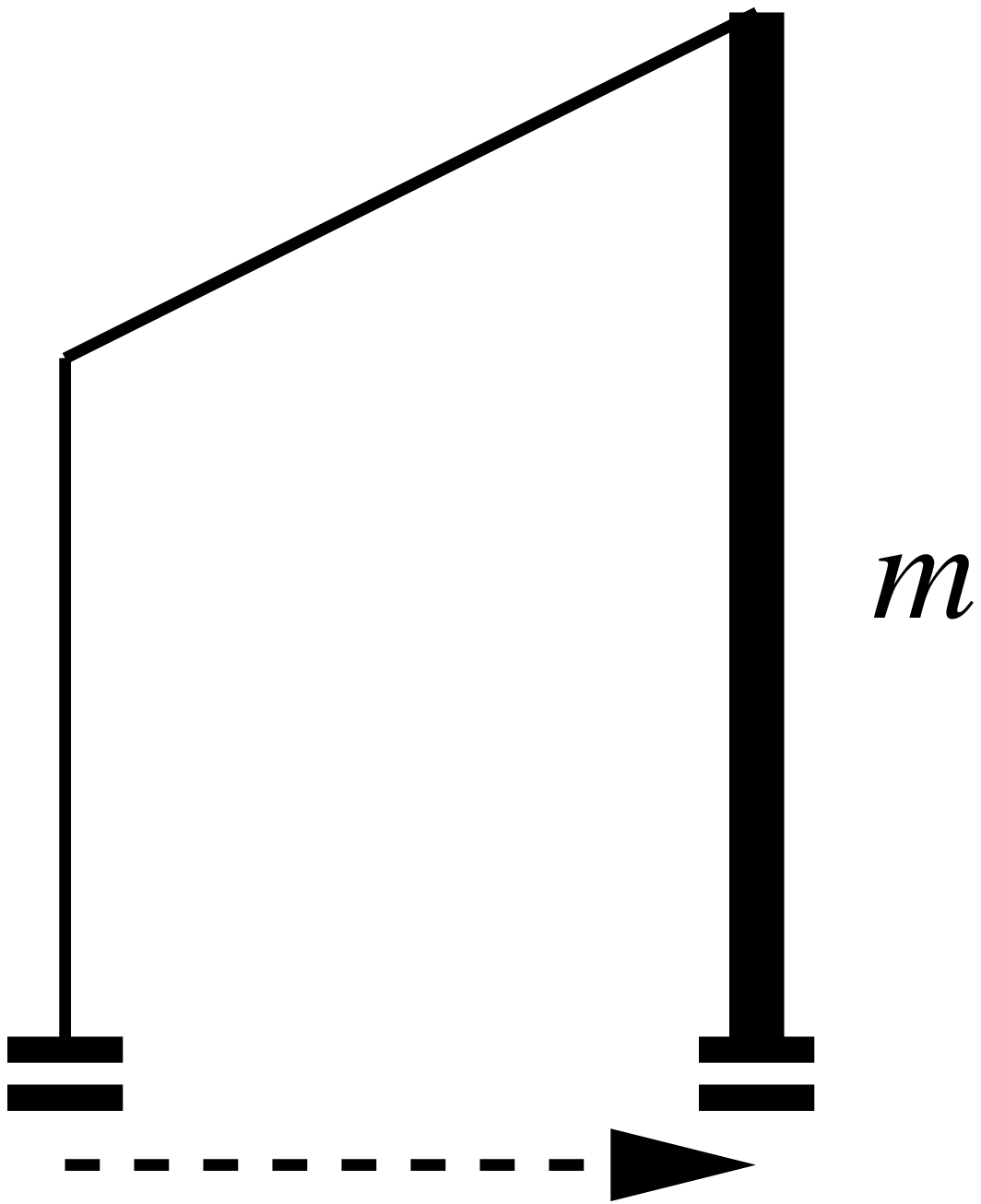}}&
W'_{p,m}&=~~\raisebox{-2pc}{\includegraphics[scale=0.15]{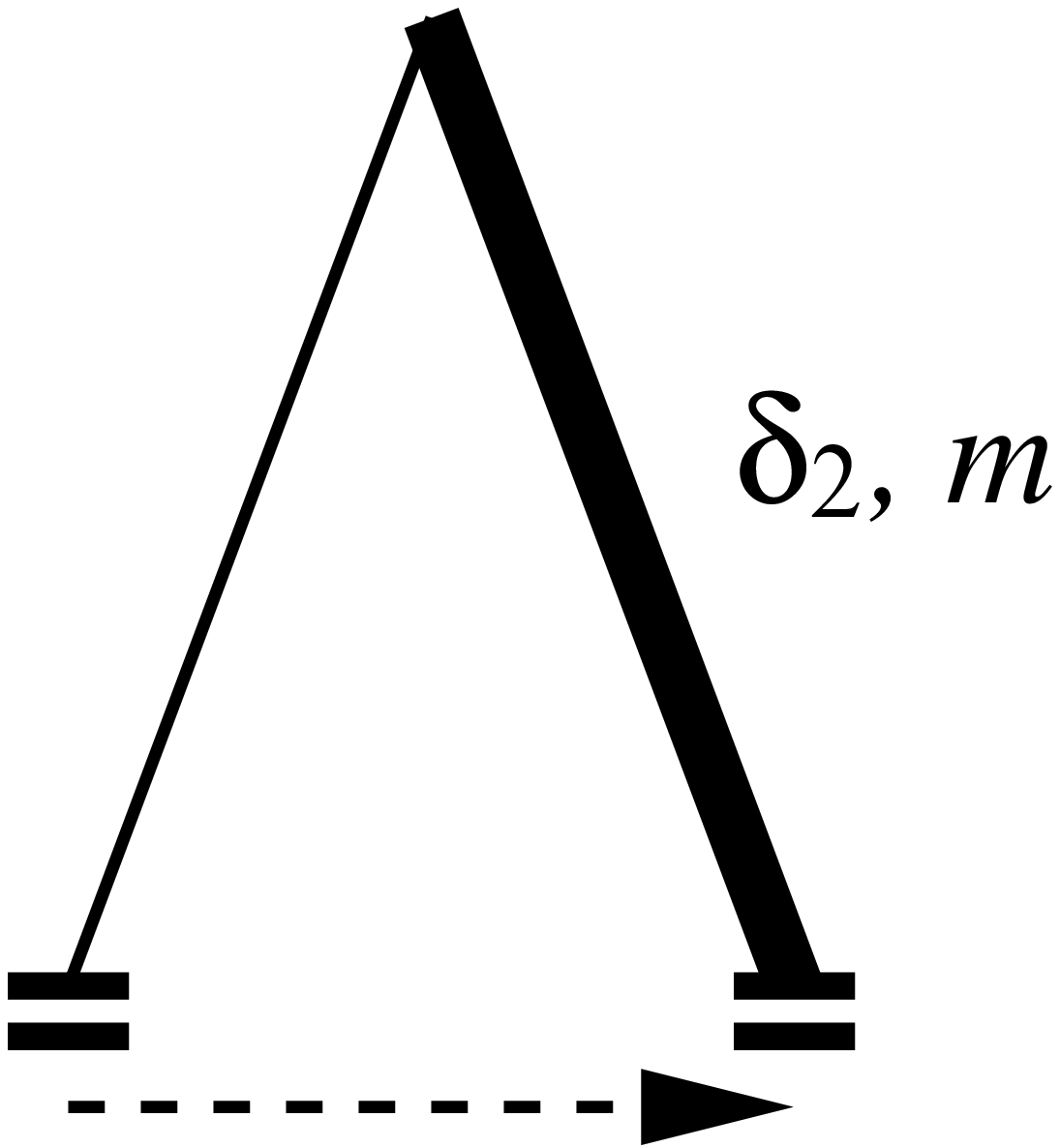}}\\[1pc]
T'_{p,m}&=~~\raisebox{-2pc}{\includegraphics[scale=0.15]{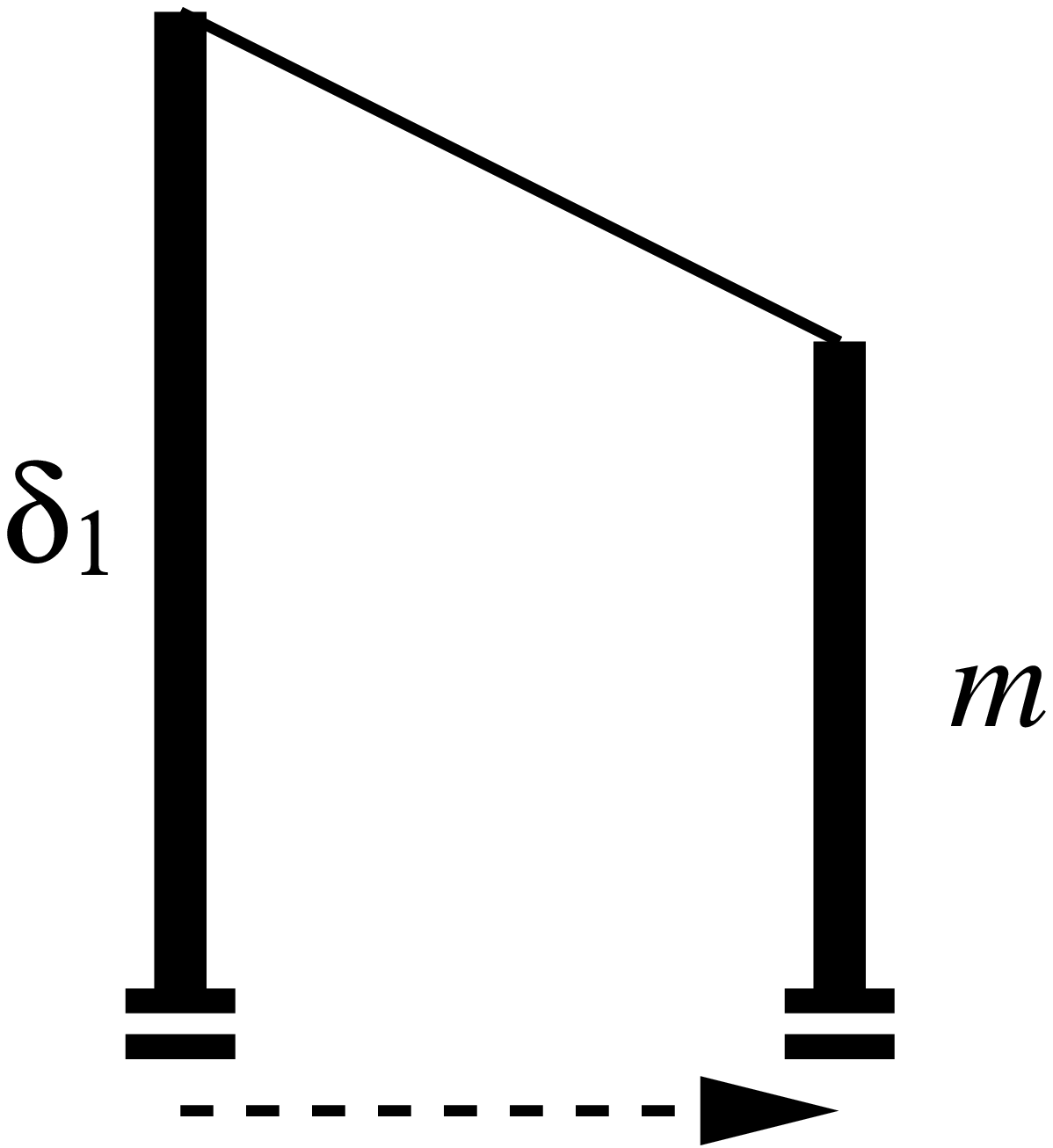}}~~+~~
 \raisebox{-2pc}{\includegraphics[scale=0.15]{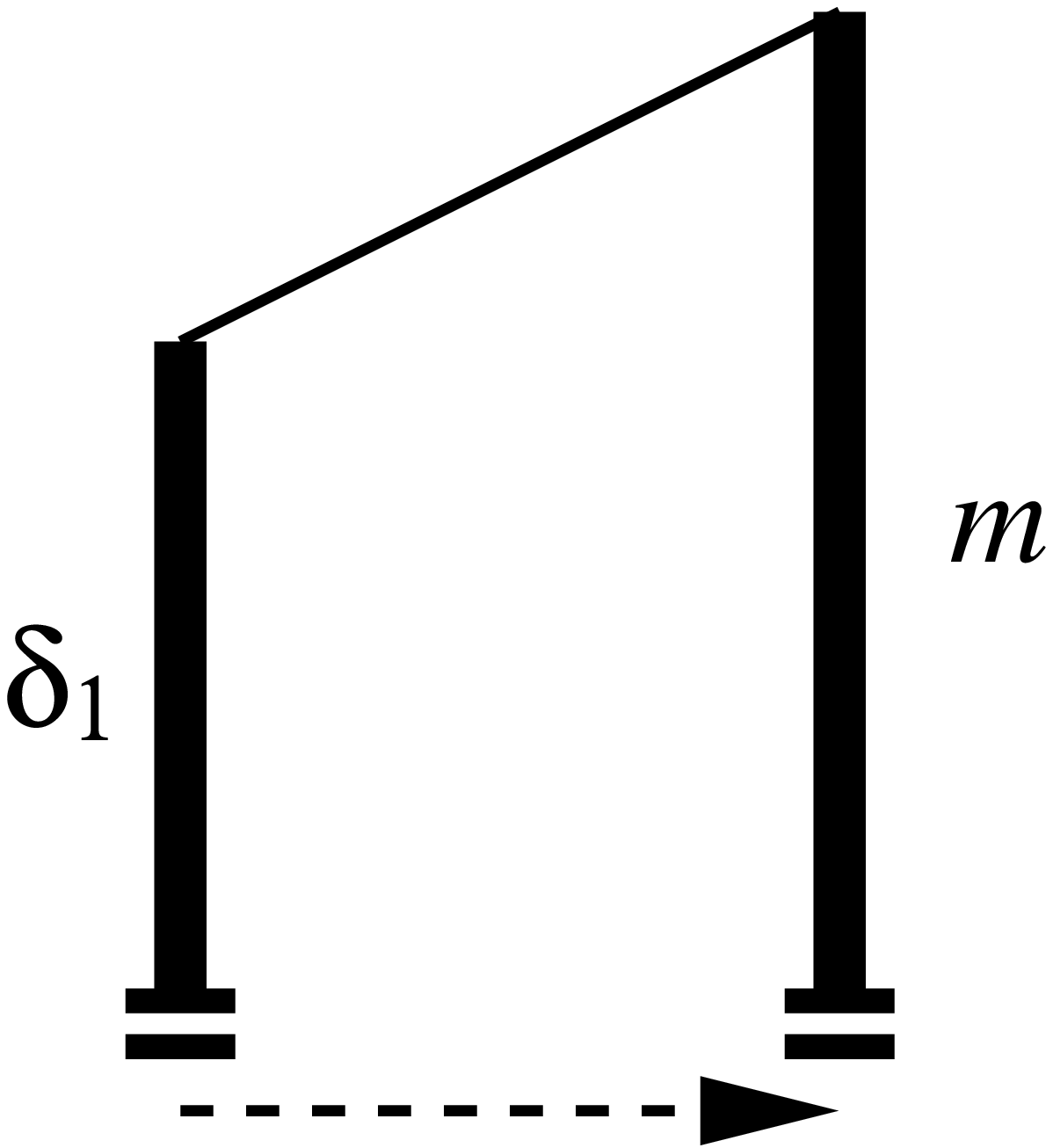}}&\qquad
W''_{p,m}&=~~\raisebox{-2pc}{\includegraphics[scale=0.15]{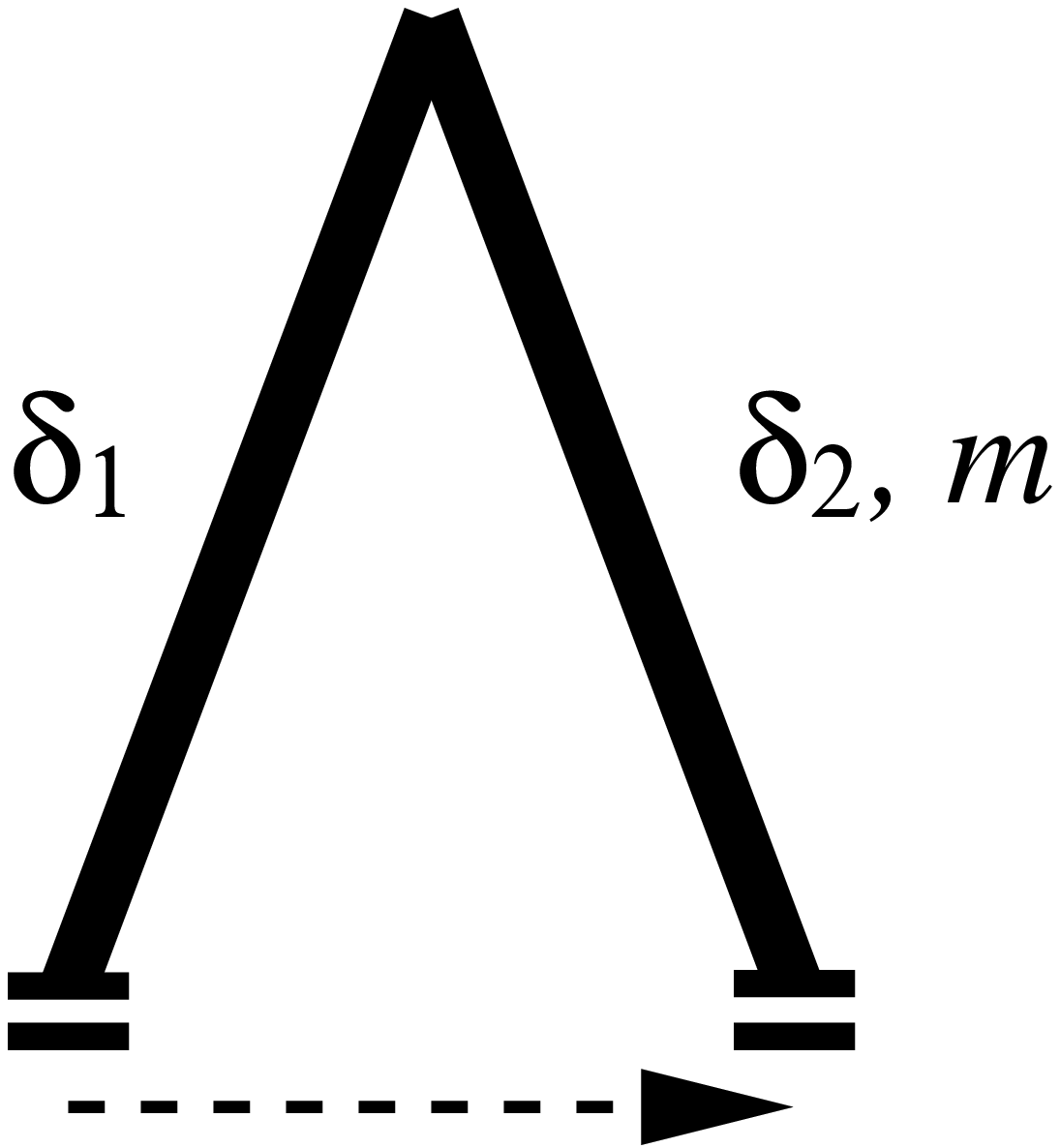}}
\end{align*}
\caption{Schematic representations of the diagram functions.  Each pair of
horizontal short line segments represents $q_p$, and the other longer line
segments represent $\varphi_p$.  A bold line segment representing
$\varphi_p(x,n)$ is weighted by the factor $m^n$ if the line segment is indexed
by $m$, and by the factor $|x_1|^\delta$ if the line segment is indexed by
$\delta$.  A dashed arrow represents the supremum over its terminal point
$\xvec\in\mZ^{d+1}$, with its initial point fixed at the origin $\ovec$.}
\end{figure}
Using the above diagram functions and Lemma~\ref{lmm:diag}, we obtain the
following:

\begin{lemma}\label{lmm:|x1|piN-bd}
For any $N\ge0$ and $m\ge0$,
\begin{align}\lbeq{|x1|piN-bd}
\sum_{(x,n)}|x_1|^{\delta_1+\delta_2}\pi_p^{\sss(N)}(x,n)\,m^n&\le(N+1)^{
 \delta_1+\delta_2}(T_{p,m})^{N-2}\bigg(\Big(N(1+T_{p,m})+T_{p,m}\Big)T_{p,
 m}W''_{p,m}\nn\\
&\qquad+N\Big((N-1)(1+T_{p,m})+3T_{p,m}\Big)T'_{p,m}W'_{p,m}\bigg).
\end{align}
\end{lemma}

\begin{proof}
First of all, by \refeq{pi0-dec}, we immediately obtain
\begin{align*}
\sum_{(x,n)}|x_1|^{\delta_1+\delta_2}\pi_p^{\sss(0)}(x,n)\,m^n\le\sum_{(x,
 n)}|x_1|^{\delta_1}\psi_p(x,n)\,|x_1|^{\delta_2}\psi_p^{\sss(m)}(x,n)\le
 W''_{p,m},
\end{align*}
as required.

Let $N\ge1$.  We denote the first coordinate of the spatial part of $\yvec_i$
by $y_{i,1}$: $\yvec_i=((y_{i,1},\dots,y_{i,d}),t_{\yvec_i})$. Similarly, we
write, e.g., $\vec\yvec_i=((\vec y_{i,1},\dots,\vec y_{i,d}),t_{\vec\yvec_i})$.
Notice that, since
\begin{align*}
|\vec y_{\sss N+1,1}|^{\delta_1}=\bigg|\sum_{j=1}^{N+1}y_{j,1}\bigg|^{\delta_1}
 \le(N+1)^{\delta_1}\max_j|y_{j,1}|^{\delta_1}\le(N+1)^{\delta_1}\sum_{j=1}^{N
 +1}|y_{j,1}|^{\delta_1},
\end{align*}
we have that, for $\vec\yvec_{N+1}=\vec\zvec_{N+1}=\xvec$,
\begin{align*}
|x_1|^{\delta_1+\delta_2}=|\vec y_{\sss N+1,1}|^{\delta_1}|\vec z_{\sss N+1,1}|
 ^{\delta_2}\le(N+1)^{\delta_1+\delta_2}\sum_{j,j'=1}^{N+1}|y_{j,1}|^{\delta_1}
 |z_{j',1}|^{\delta_2}.
\end{align*}
By this inequality and using \refeq{piN-dec1}--\refeq{piN-dec2}, we obtain
\begin{align}\lbeq{|x1|piN-bd1}
\sum_{(x,n)}|x_1|^{\delta_1+\delta_2}\pi_p^{\sss(N)}(x,n)\,m^n\le
 (N+1)^{\delta_1+\delta_2}\sum_{j'=1}^{N+1}S_{j'},
\end{align}
where
\begin{align*}
S_1=\sum_{j=1}^{N+1}\sum_{\substack{\yvec_1,\dots,\yvec_{N+1}\\ \zvec_1,
 \dots,\zvec_{N+1}\\ (\vec\yvec_{N+1}=\vec\zvec_{N+1})\\ (t_{\yvec_1}\ge
 t_{\zvec_1})}}&|y_{j,1}|^{\delta_1}\varphi_p(\yvec_1)\,|z_{1,1}|^{\delta_2}
 \varphi_p^{\sss(m)}(\zvec_1)\prod_{i=1}^N\tilde\Lambda_p^{\sss(m)}(\vec
 \yvec_i,\vec\zvec_i;\vec\yvec_{i+1},\vec\zvec_{i+1}),
\end{align*}
and, for $j'>1$,
\begin{align*}
S_{j'}=\sum_{j=1}^{N+1}\sum_{\substack{\yvec_1,\dots,\yvec_{N+1}\\ \zvec_1,
 \dots,\zvec_{N+1}\\ (\vec\yvec_{N+1}=\vec\zvec_{N+1})}}&|y_{j,1}|^{\delta_1}
 \varphi_p(\yvec_1)\,\varphi_p^{\sss(m)}(\zvec_1)\,\varphi_p(\yvec_1-\zvec_1)
 \bigg(\prod_{i=2}^{j'-1}\Lambda_p^{\sss(m)}(\vec\yvec_{i-1},\vec\zvec_{i-1};
 \vec\yvec_i,\vec\zvec_i)\bigg)\\
&\times\psi_p(\yvec_{j'})\,|z_{j',1}|^{\delta_2}\psi_p^{\sss(m)}(\zvec_{j'})
 \bigg(\prod_{i=j'}^N\tilde\Lambda_p^{\sss(m)}(\vec\yvec_i,\vec\zvec_i;\vec
 \yvec_{i+1},\vec\zvec_{i+1})\bigg).
\end{align*}

It remains to estimate each $S_{j'}$.  To do so, we follow the same line of
argument in \cite[Section~2]{s??}.  Here, we explain in detail how to estimate
$S_1$.  First we note that, by translation-invariance,
\begin{align}
\sup_{\yvec}\sum_{\wvec,\xvec}\tilde\Lambda_p^{\sss(m)}(\ovec,\wvec;\xvec,
 \xvec+\yvec)&\le T_{p,m},\lbeq{tLambda-bd1}\\
\sup_{\yvec}\sum_{\wvec,\xvec}|y_1|^{\delta_1}\tilde\Lambda_p^{\sss(m)}(\ovec,
 \wvec;\xvec,\xvec+\yvec)&\le T'_{p,m}.\lbeq{tLambda-bd2}
\end{align}
Then, by repeated use of translation-invariance, the contribution to $S_1$ from
$j=1$ is bounded as
\begin{align}\lbeq{S11-bd}
&\sum_{\substack{\yvec_1,\dots,\yvec_{N+1}\\ \zvec_1,\dots,\zvec_{N+1}\\
 (\vec\yvec_{N+1}=\vec\zvec_{N+1})\\ (t_{\yvec_1}\ge t_{\zvec_1})}}|y_{1,
 1}|^{\delta_1}\varphi_p(\yvec_1)\,|z_{1,1}|^{\delta_2}\varphi_p^{\sss(m)}
 (\zvec_1)\prod_{i=1}^N\tilde\Lambda_p^{\sss(m)}(\vec\yvec_i,\vec\zvec_i;
 \vec\yvec_{i+1},\vec\zvec_{i+1})\nn\\
&=\sum_{\wvec,\xvec}\tilde\Lambda_p^{\sss(m)}(\ovec,\wvec;\xvec,\xvec)
 \sum_{\substack{\yvec_1,\dots,\yvec_N\\ \zvec_1,\dots,\zvec_N\\ (\vec
 \zvec_N=\vec\yvec_N+\wvec)\\ (t_{\yvec_1}\ge t_{\zvec_1})}}|y_{1,
 1}|^{\delta_1}\varphi_p(\yvec_1)\,|z_{1,1}|^{\delta_2}\varphi_p^{\sss(m)}
 (\zvec_1)\prod_{i=1}^{N-1}\tilde\Lambda_p^{\sss(m)}(\vec\yvec_i,\vec\zvec_i;
 \vec\yvec_{i+1},\vec\zvec_{i+1})\nn\\
&\le T_{p,m}~\sup_{\wvec}\sum_{\substack{\yvec_1,\dots,\yvec_N\\ \zvec_1,
 \dots,\zvec_N\\ (\vec\zvec_N=\vec\yvec_N+\wvec)\\ (t_{\yvec_1}\ge
 t_{\zvec_1})}}|y_{1,1}|^{\delta_1}\varphi_p(\yvec_1)\,|z_{1,1}|^{\delta_2}
 \varphi_p^{\sss(m)}(\zvec_1)\prod_{i=1}^{N-1}\tilde\Lambda_p^{\sss(m)}(\vec
 \yvec_i,\vec\zvec_i;\vec\yvec_{i+1},\vec\zvec_{i+1})\nn\\
&\le(T_{p,m})^2~\sup_{\wvec}\sum_{\substack{\yvec_1,\dots,\yvec_{N-1}\\
 \zvec_1,\dots,\zvec_{N-1}\\ (\vec\zvec_{N-1}=\vec\yvec_{N-1}+\wvec)\\ (t_{\yvec_1}\ge t_{\zvec_1})}}|y_{1,1}|^{\delta_1}\varphi_p(\yvec_1)\,
 |z_{1,1}|^{\delta_2}\varphi_p^{\sss(m)}(\zvec_1)\prod_{i=1}^{N-2}\tilde
 \Lambda_p^{\sss(m)}(\vec\yvec_i,\vec\zvec_i;\vec\yvec_{i+1},\vec\zvec_{i
 +1})\nn\\
&~\:\vdots\nn\\
&\le(T_{p,m})^N~\sup_{\wvec:t_{\wvec}\ge0}\sum_{\substack{\yvec_1,\zvec_1\\
 (\yvec_1=\zvec_1+\wvec)}}|y_{1,1}|^{\delta_1}\varphi_p(\yvec_1)\,|z_{1,
 1}|^{\delta_2}\varphi_p^{\sss(m)}(\zvec_1)\nn\\
&\le(T_{p,m})^NW''_{p,m},
\end{align}
where we have used $\varphi_p(\xvec)\le\delta_{\xvec,\ovec}+\psi_p(\xvec)$. Similarly, the contribution to $S_1$ from $j>1$ is bounded as
\begin{align*}
&\sum_{\substack{\yvec_1,\dots,\yvec_{N+1}\\ \zvec_1,\dots,\zvec_{N+1}\\
 (\vec\yvec_{N+1}=\vec\zvec_{N+1})\\ (t_{\yvec_1}\ge t_{\zvec_1})}}|y_{j,
 1}|^{\delta_1}\varphi_p(\yvec_1)\,|z_{1,1}|^{\delta_2}\varphi_p^{\sss(m)}
 (\zvec_1)\prod_{i=1}^N\tilde\Lambda_p^{\sss(m)}(\vec\yvec_i,\vec\zvec_i;
 \vec\yvec_{i+1},\vec\zvec_{i+1})\\
&\le(T_{p,m})^{N-1}T'_{p,m}~\sup_{\wvec:t_{\wvec\ge0}}\sum_{\substack{\yvec_1,
 \zvec_1\\ (\yvec_1=\zvec_1+\wvec)}}\varphi_p(\yvec_1)\,|z_{1,1}|^{\delta_2}
 \varphi_p^{\sss(m)}(\zvec_1)\\
&\le(T_{p,m})^{N-1}T'_{p,m}W'_{p,m}.
\end{align*}
Therefore,
\begin{align}\lbeq{S1-bd}
S_1\le N(T_{p,m})^{N-1}T'_{p,m}W'_{p,m}+(T_{p,m})^NW''_{p,m}.
\end{align}

To estimate $S_{j'}$ for $j'>1$, we first use
\refeq{tLambda-bd1}--\refeq{tLambda-bd2}.  For example, the contribution from
$j=j'$ is bounded, similarly to \refeq{S11-bd}, as
\begin{align}\lbeq{Sjj-bd}
&\sum_{\substack{\yvec_1,\dots,\yvec_{N+1}\\ \zvec_1,\dots,\zvec_{N+1}\\ (\vec
 \yvec_{N+1}=\vec\zvec_{N+1})}}\varphi_p(\yvec_1)\,\varphi_p^{\sss(m)}(\zvec_1)
 \,\varphi_p(\yvec_1-\zvec_1)\bigg(\prod_{i=2}^{j'-1}\Lambda_p^{\sss(m)}(\vec
 \yvec_{i-1},\vec\zvec_{i-1};\vec\yvec_i,\vec\zvec_i)\bigg)\nn\\
&\hskip5pc\times|y_{j',1}|^{\delta_1}\psi_p(\yvec_{j'})\,|z_{j',1}|^{\delta_2}
 \psi_p^{\sss(m)}(\zvec_{j'})\bigg(\prod_{i=j'}^N\tilde\Lambda_p^{\sss(m)}(\vec
 \yvec_i,\vec\zvec_i;\vec\yvec_{i+1},\vec\zvec_{i+1})\bigg)\nn\\
&\le(T_{p,m})^{N+1-j'}~\sup_{\xvec}\sum_{\substack{\yvec_1,\dots,\yvec_{j
 '}\\ \zvec_1,\dots,\zvec_{j'}\\ (\vec\zvec_{j'}=\vec\yvec_{j'}+\xvec)}}
 \varphi_p(\yvec_1)\,\varphi_p^{\sss(m)}(\zvec_1)\,\varphi_p(\yvec_1-\zvec_1)
 \nn\\
&\hskip5pc\times\bigg(\prod_{i=2}^{j'-1}\Lambda_p^{\sss(m)}(\vec\yvec_{i-1},
 \vec\zvec_{i-1};\vec\yvec_i,\vec\zvec_i)\bigg)|y_{j',1}|^{\delta_1}\psi_p
 (\yvec_{j'})\,|z_{j',1}|^{\delta_2}\psi_p^{\sss(m)}(\zvec_{j'}).\nn\\
&\le(T_{p,m})^{N+1-j'}W''_{p,m}~\sup_{\xvec}\sum_{\substack{\yvec_1,
 \dots,\yvec_{j'-1}\\ \zvec_1,\dots,\zvec_{j'-1}\\ (\vec\zvec_{j'-1}=\vec
 \yvec_{j'-1}+\xvec)}}\varphi_p(\yvec_1)\,\varphi_p^{\sss(m)}(\zvec_1)\,
 \varphi_p(\yvec_1-\zvec_1)\nn\\
&\hskip5pc\times\bigg(\prod_{i=2}^{j'-1}\Lambda_p^{\sss(m)}(\vec\yvec_{i-1},
 \vec\zvec_{i-1};\vec\yvec_i,\vec\zvec_i)\bigg).
\end{align}
Notice that
\begin{align*}
\sup_{\xvec}\sum_{\yvec,\zvec}\Lambda_p^{\sss(m)}(\ovec,\xvec;\yvec,
 \zvec)\le T_{p,m},&&&&
\sup_{\xvec}\sum_{\yvec,\zvec}|y_1|^{\delta_1}\Lambda_p^{\sss(m)}(\ovec,
 \xvec;\yvec,\zvec)\le T'_{p,m}.
\end{align*}
By repeated use of translation-invariance, we obtain
\begin{align*}
\refeq{Sjj-bd}\le(1+T_{p,m})(T_{p,m})^{N-1}W''_{p,m}.
\end{align*}
It is not hard to see that the contribution from $j$ not being either $j'$ or
1, which is possible only if $N\ge2$, is bounded by
$(1+T_{p,m})(T_{p,m})^{N-2}T'_{p,m}W'_{p,m}$, and the contribution from $j=1$
is bounded by $2(T_{p,m})^{N-1}T'_{p,m}W'_{p,m}$. Therefore, for $j'>1$,
\begin{align}\lbeq{Sj-bd}
S_{j'}&\le\big((N-1)(1+T_{p,m})+2T_{p,m}\big)(T_{p,m})^{N-2}T'_{p,m}
 W'_{p,m}+(1+T_{p,m})(T_{p,m})^{N-1}W''_{p,m}.
\end{align}

The proof of \refeq{|x1|piN-bd} is completed by assembling
\refeq{|x1|piN-bd1}, \refeq{S1-bd} and \refeq{Sj-bd}.
\end{proof}

\subsection{Integral representation of fractional-power functions}\label{ss:prelim}
In this subsection, we use an integral representation of $a^\delta$ for $a>0$
and $\delta\in(0,2)$ to bound the diagram functions $T'_{p,m}$, $W'_{p,m}$ and
$W''_{p,m}$.

First we note that, for $\delta\in(0,2)$,
\begin{align*}
K_\delta=\int_0^\infty\frac{1-\cos t}{t^{1+\delta}}\;\textrm{d}t
\end{align*}
is a positive finite constant.  Replacing $t$ by $u=t/a$ with $a>0$, we obtain
\begin{align}\lbeq{anyr}
a^\delta=\frac1{K_\delta}\int_0^\infty\frac{1-\cos(ua)}{u^{1+\delta}}\;
 \textrm{d}u\le\frac1{K_\delta}\bigg(\frac2\delta+\int_0^1\frac{1-\cos(ua)}
 {u^{1+\delta}}\;\textrm{d}u\bigg),
\end{align}
which is the key inequality.

To describe bounds on $T'_{p,m}$, $W'_{p,m}$ and $W''_{p,m}$ below, we define
\begin{align*}
\hat Y_k(l,z)=|\Delta_k\hat D(l)|\,|\hat\varphi_p(l,z)|+|\Delta_k\hat\varphi_p
 (l,z)|,
\end{align*}
and, by denoting $\vec u=(u,0,\dots,0)\in[-\pi,\pi]^d$,
\begin{align}
\hat I_1(u)&=\int_{[-\pi,\pi]^d}\frac{{\rm d}^dl}{(2\pi)^d}\int_{-\pi}^\pi
 \frac{{\rm d}\theta}{2\pi}\;\hat Y_{\vec u}(l,e^{i\theta})\,|\hat\varphi_p
 (l,e^{i\theta})|\,|\hat\varphi_p(l,me^{i\theta})|,\lbeq{I1-def}\\
\hat I_2(v)&=\int_{[-\pi,\pi]^d}\frac{{\rm d}^dl}{(2\pi)^d}\int_{-\pi}^\pi
 \frac{{\rm d}\theta}{2\pi}\;|\hat\varphi_p(l,e^{i\theta})|\,\hat Y_{\vec v}(l,me^{i\theta}),\lbeq{I2-def}\\
\hat I_3(u)&=\int_{[-\pi,\pi]^d}\frac{{\rm d}^dl}{(2\pi)^d}\int_{-\pi}^\pi
 \frac{{\rm d}\theta}{2\pi}\;\hat Y_{\vec u}(l,e^{i\theta})\,|\hat\varphi_p
 (l,me^{i\theta})|,\lbeq{I3-def}\\
\hat I_4(u,v)&=\int_{[-\pi,\pi]^d}\frac{{\rm d}^dl}{(2\pi)^d}\int_{-
 \pi}^\pi\frac{{\rm d}\theta}{2\pi}\;\hat Y_{\vec u}(l,e^{i\theta})\,\hat Y_{\vec v}(l,me^{i\theta}).\lbeq{I4-def}
\end{align}
Taking the Fourier-Laplace transform of \refeq{T'-def}--\refeq{W''-def} (also
recalling \refeq{DeltaFL}) and using \refeq{anyr}, we obtain the following:

\begin{lemma}\label{lmm:diagfurther}
For any $p\in(0,\pc)$ and $m\in[0,m_p)$,
\begin{align}
T'_{p,m}&\le\frac1{K_{\delta_1}}\bigg(\frac2{\delta_1}T_{p,m}+5p^2m\int_0^1
 \frac{{\rm d}u}{u^{1+\delta_1}}\;\hat I_1(u)\bigg),\lbeq{T'-bd}\\
W'_{p,m}&\le\frac1{K_{\delta_2}}\bigg(\frac1{\delta_2}T_{p,m}+\frac{5p^2m}2
 \int_0^1\frac{{\rm d}v}{v^{1+\delta_2}}\;\hat I_2(v)\bigg),\lbeq{W'-bd}\\
W''_{p,m}&\le\frac1{K_{\delta_1}}\bigg(\frac2{\delta_1}W'_{p,m}+\frac{5p^2m}
 {K_{\delta_2}\delta_2}\int_0^1\frac{{\rm d}u}{u^{1+\delta_1}}\;\hat I_3
 (u)+\frac{25p^2m}{4K_{\delta_2}}\int_0^1\frac{{\rm d}u}{u^{1+
 \delta_1}}\int_0^1\frac{{\rm d}v}{v^{1+\delta_2}}\;\hat I_4(u,v)\bigg).
 \lbeq{W''-bd}
\end{align}
\end{lemma}

\begin{proof}
We only prove \refeq{W'-bd}, since the other two inequalities can be proved in the same way.

First we use \refeq{anyr} to bound $|y_1-x_1|^{\delta_2}$ in \refeq{W'-def}.  The first term in \refeq{W'-bd} is due to the first term in \refeq{anyr} and the trivial inequality
\begin{align*}
\sum_{\yvec}\psi_p(\yvec)\,\psi_p^{\sss(m)}(\yvec-\xvec)\le\frac12T_{p,m}.
\end{align*}
To complete the proof of \refeq{W'-bd}, it thus remains to show
\begin{align}\lbeq{W'2nd-bd1}
\sum_{\yvec}\psi_p(\yvec+\xvec)\,\big(1-\cos(vy_1)\big)\psi_p^{\sss(m)}(\yvec)
 \le\frac{5p^2m}2\hat I_2(v).
\end{align}
However, since
$1-\cos\sum_{j=1}^Jt_j\le(2J+1)\sum_{j=1}^J(1-\cos t_j)$ (cf., \cite[(4.50)]{s06}), we have
\begin{align*}
\big(1-\cos(vy_1)\big)\psi_p(\yvec)&\equiv\big(1-\cos(vy_1)\big)(q_p*\varphi_p)
 (\yvec)\\
&\le5p\sum_{\wvec}\big(1-\cos(vw_1)\big)\Big(D(w)\,\varphi_p(\yvec-\wvec)
 +\varphi_p(\wvec)\,D(y-w)\Big).
\end{align*}
Applying this to the left-hand side of \refeq{W'2nd-bd1}, then taking the Fourier-Laplace transform and using $|\hat\varphi_p(l,e^{i\theta})|=|\hat\varphi_p(l,e^{-i\theta})|$, we obtain \refeq{W'2nd-bd1}.  This completes the proof of \refeq{W'-bd}.
\end{proof}

\subsection{Bounds on the diagram functions}\label{ss:bounds}
In this subsection, we complete the proof of Proposition~\ref{prp:key} using the following lemma:

\begin{lemma}\label{lmm:Ii-bds}
Let $\alpha>0$ and $d>2(\alpha\wedge2)$, and choose $\delta$ as in \refeq{delta-def} and $\delta_1,\delta_2\in(0,2)$ as
\begin{align*}
\delta<\delta_1<\alpha\wedge2\wedge\big(d-2(\alpha\wedge2)\big),&&
\delta_2=\alpha\wedge2+\delta-\delta_1.
\end{align*}
Then,
\begin{align}\lbeq{unformbds}
T_{p,m}=O(\lambda),&&
 \left.\begin{array}{c}
 T'_{p,m}\\ W'_{p,m}\\ W''_{p,m}
 \end{array}\!\right\}=O(1),
\end{align}
uniformly in $p\in(0,\pc)$ and $m\in[0,m_p)$.
\end{lemma}

\begin{proof}[Proof of Proposition~\ref{prp:key}]
First, by Lemmas~\ref{lmm:elementary} and \ref{lmm:|x1|piN-bd} with $r=\delta_1+\delta_2\equiv\alpha\wedge2+\delta<4$, we obtain that, for any $p\in(0,\pc]$,
\begin{align}\lbeq{|x|pi-proof1}
&\sum_{(x,n)}|x|^{\alpha\wedge2+\delta}|\pi_p(x,n)|\,m_p^n\nn\\
&\le d^3\sum_{N=0}^\infty(N+1)^{\delta_1+\delta_2}(T_{p,m_p})^{N-2}\bigg(\Big(
 N(1+T_{p,m_p})+T_{p,m_p}\Big)T_{p,m_p}W''_{p,m_p}\nn\\
&\hskip9pc+N\Big((N-1)(1+T_{p,m_p})+3T_{p,m_p}\Big)T'_{p,m_p}W'_{p,m_p}\bigg).
\end{align}
Since the diagram functions \refeq{T-def}--\refeq{W''-def} are increasing in $m\ge0$ for every $p\ge0$ and in $p\ge0$ for every $m\ge0$, the uniform bounds in \refeq{unformbds} imply that these diagram functions at $m=m_p$ obey the same bounds uniformly in $p\in(0,\pc]$.  Therefore, the right-hand side of \refeq{|x|pi-proof1} is convergent, if $\lambda$ is sufficiently small.  This completes the proof of Proposition~\ref{prp:key}.
\end{proof}

\begin{proof}[Proof of Lemma~\ref{lmm:Ii-bds}]
It is not hard to extend \cite[Lemma~4.1]{cs08} to show that
$T_{p,m}=O(\lambda)$ uniformly in $p\in(0,\pc)$ and $m\in[0,m_p)$.  Recall
Lemma~\ref{lmm:diagfurther}.  To complete the proof of Lemma~\ref{lmm:Ii-bds},
it thus suffices to show that the integrals in \refeq{T'-bd}--\refeq{W''-bd} of
$\hat I_1,\dots,\hat I_4$ are bounded uniformly in $p\in(0,\pc)$ and
$m\in[0,m_p)$.

The integrals of $\hat I_2$ and $\hat I_3$ are easy and can be estimated
similarly.  For example, by \refeq{varphi-bd}--\refeq{Deltavarphi-bd} and
$|\Delta_{\vec v}\hat D(l)|\le2\big(1-\hat D(\vec v)\big)$ (cf.,
\refeq{DeltaFL}),
\begin{align*}
\hat I_2(v)&=\int\frac{\textrm{d}^dl}{(2\pi)^d}\int\frac{\textrm{d}\theta}{2
 \pi}\;|\hat\varphi_p(l,e^{i\theta})|\Big(|\Delta_{\vec v}\hat D(l)|\,|\hat
 \varphi_p(l,me^{i\theta})|+|\Delta_{\vec v}\hat\varphi_p(l,me^{i\theta})|
 \Big)\\
&\le O(1-\hat D(\vec v))\int\frac{\textrm{d}^dl}{(2\pi)^d}\int\frac{\textrm{d}
 \theta}{2\pi}\;\frac1{|\theta|+1-\hat D(l)}\bigg(\frac1{|\theta|+1-\hat D(l)}
 \\
&\hskip5pc+\sum_{(j,j')=(0,\pm1),(1,-1)}\frac1{(|\theta|+1-\hat D(l+j\vec v))
 (|\theta|+1-\hat D(l+j'\vec v))}\bigg)
\end{align*}
holds uniformly in $p\in(0,\pc)$ and $m\in[0,m_p)$.  Using the H\"older inequality twice and the translation-invariance of $D$, we have
\begin{align*}
&\int\frac{\textrm{d}\theta}{2\pi}\;\frac1{|\theta|+1-\hat D(l)}\frac1
 {(|\theta|+1-\hat D(l+j\vec v))(|\theta|+1-\hat D(l+j'\vec v))}\\
&\le\Bigg(\int\frac{\textrm{d}\theta}{2\pi}\;\frac1{|\theta|+1-\hat D(l)}
 \bigg(\frac1{|\theta|+1-\hat D(l+j\vec v)}\bigg)^2\Bigg)^{1/2}\\
&\quad\times\Bigg(\int\frac{\textrm{d}\theta}{2\pi}\;\frac1{|\theta|+1-\hat
 D(l)}\bigg(\frac1{|\theta|+1-\hat D(l+j'\vec v)}\bigg)^2\Bigg)^{1/2}\\
&\le\Bigg(\int\frac{\textrm{d}\theta}{2\pi}\;\bigg(\frac1{|\theta|+1-\hat
 D(l)}\bigg)^3\Bigg)^{1/6}\Bigg(\int\frac{\textrm{d}\theta}{2\pi}\;\bigg(
 \frac1{|\theta|+1-\hat D(l+j\vec v)}\bigg)^3\Bigg)^{1/3}\\
&\quad\times\Bigg(\int\frac{\textrm{d}\theta}{2\pi}\;\bigg(\frac1{|\theta|
 +1-\hat D(l)}\bigg)^3\Bigg)^{1/6}\Bigg(\int\frac{\textrm{d}\theta}{2\pi}\;
 \bigg(\frac1{|\theta|+1-\hat D(l+j'\vec v)}\bigg)^3\Bigg)^{1/3}\\
&=\int\frac{\textrm{d}\theta}{2\pi}\;\bigg(\frac1{|\theta|+1-\hat D(l)}
 \bigg)^3.
\end{align*}
Since, by \refeq{valpha},
\begin{align*}
\int\frac{\textrm{d}^dl}{(2\pi)^d}\int\frac{\textrm{d}\theta}{2\pi}\;\bigg(
 \frac1{|\theta|+1-\hat D(l)}\bigg)^3\le\int\frac{\textrm{d}^dl}{(2\pi)^d}\;
 \frac{O(1)}{(1-\hat D(l))^2}<\infty
\end{align*}
holds for $d>2(\alpha\wedge2)$, we conclude that, for
$\delta_2=\alpha\wedge2-(\delta_1-\delta)<\alpha\wedge2$,
\begin{align*}
\int_0^1\frac{\textrm{d}v}{v^{1+\delta_2}}\;\hat I_2(v)\le\int_0^1
 \frac{\textrm{d}v}{v^{1+\delta_2}}\;O\big(1-\hat D(\vec v)\big)<\infty,
\end{align*}
as required.

Next, we consider the integral of $\hat I_1$.  In fact, we only need consider
the contribution from $|\Delta_{\vec u}\hat\varphi_p(l,e^{i\theta})|$ in $\hat
Y_{\vec u}(l,e^{i\theta})$ of \refeq{I1-def}, because the contribution from the
other term in $\hat Y_{\vec u}(l,e^{i\theta})$ can be estimated similarly to
the integral of $\hat I_2$, as explained above.  Using
\refeq{varphi-bd}--\refeq{Deltavarphi-bd} and ignoring some factors of
$|\theta|$, we obtain
\begin{align}\lbeq{I1-bd1}
&\int\frac{\textrm{d}^dl}{(2\pi)^d}\int\frac{\textrm{d}\theta}{2\pi}\;|\Delta_{
 \vec u}\hat\varphi_p(l,e^{i\theta})|\,|\hat\varphi_p(l,e^{i\theta})|\,|\hat
 \varphi_p(l,me^{i\theta})|\nn\\
&\le\sum_{(j,j')}\int\frac{\textrm{d}^dl}{(2\pi)^d}\;\frac{O(1-\hat D(\vec u))}
 {(1-\hat D(l+j\vec u))(1-\hat D(l+j'\vec u))}\int\frac{\textrm{d}\theta}{2\pi}
 \bigg(\frac1{|\theta|+1-\hat D(l)}\bigg)^2\nn\\
&\le O\big(1-\hat D(\vec u)\big)\sum_{(j,j')}\int\frac{\textrm{d}^dl}{(2\pi)^d}
 \;\frac1{(1-\hat D(l+j\vec u))(1-\hat D(l+j'\vec u))(1-\hat D(l))},
\end{align}
where $\sum_{(j,j')}$ is the sum over $(j,j')=(0,\pm1),(1,-1)$.  By the
translation-invariance and $\Zd$-symmetry of $D$, the integral for
$(j,j')=(0,\pm1)$ equals
\begin{align}\lbeq{J-def}
\hat J(u)=\int_{[-\pi,\pi]^d}\frac{\textrm{d}^dl}{(2\pi)^d}\;\frac1{(1-\hat D
 (l))^2(1-\hat D(l-\vec u))}.
\end{align}
Moreover, by the Schwarz inequality, the integral for $(j,j')=(1,-1)$ is
bounded by
\begin{align*}
&\bigg(\int\frac{\textrm{d}^dl}{(2\pi)^d}\;\frac1{(1-\hat D(l+\vec u))^2(1-\hat
 D(l))}\bigg)^{1/2}\\
&\times\bigg(\int\frac{\textrm{d}^dl}{(2\pi)^d}\;\frac1{(1-\hat D(l-\vec u))^2
 (1-\hat D(l))}\bigg)^{1/2}=\hat J(u).
\end{align*}
Therefore,
\begin{align}\lbeq{I1-bd2}
\refeq{I1-bd1}\le O\big(1-\hat D(\vec u)\big)\hat J(u).
\end{align}

Now we show
\begin{align}\lbeq{Jbd}
\hat J(u)\le O\big(u^{(d-3(\alpha\wedge2))\wedge0}\big),
\end{align}
which is sufficient for the integral of $\hat I_1$ to be convergent for
$\delta_1<\alpha\wedge2\wedge(d-2(\alpha\wedge2))$.  Let
\begin{align}
R_1&=\big\{l\in[-\pi,\pi]^d:|l|\ge\tfrac32u\big\},\\
R_2&=\big\{l\in[-\pi,\pi]^d:|l|\le\tfrac32u,~|l|\le|l-\vec u|\big\},\\
R_3&=\big\{l\in[-\pi,\pi]^d:|l|\le\tfrac32u,~|l-\vec u|\le|l|\big\}.
\end{align}
Notice that $1-\hat D(l-\vec u)\ge O(|l-\vec u|^{\alpha\wedge2})$ for any
$l\in[-\pi,\pi]^d$ and $u\in[0,1]$ (cf., \cite[Proposition~1.1]{cs08}).  Since
$|l-\vec u|\ge|l|-u\ge\frac13|l|$ for $l\in R_1$, we have
\begin{align*}
\int_{R_1}\frac{\textrm{d}^dl}{(2\pi)^d}\;\frac1{(1-\hat D(l))^2(1-\hat D(l
 -\vec u))}&\le O(1)\int_{R_1}\frac{\textrm{d}^dl}{|l|^{3(\alpha\wedge2)}}\le
 O\big(u^{(d-3(\alpha\wedge2))\wedge0}\big).
\end{align*}
Moreover, since $|l-\vec u|\ge\frac{u}2$ for $l\in R_2$ and $|l|\ge\frac{u}2$
for $l\in R_3$, we have
\begin{align*}
\int_{R_2}\frac{\textrm{d}^dl}{(2\pi)^d}\;\frac1{(1-\hat D(l))^2(1-\hat D(l
 -\vec u))}&\le O(u^{-\alpha\wedge2})\int_{|l|\le\frac32u}\frac{\textrm{d}^dl}
 {|l|^{2(\alpha\wedge2)}}\le O\big(u^{d-3(\alpha\wedge2)}\big),
\end{align*}
and
\begin{align*}
\int_{R_3}\frac{\textrm{d}^dl}{(2\pi)^d}\;\frac1{(1-\hat D(l))^2(1-\hat D(l
 -\vec u))}&\le O(u^{-2(\alpha\wedge2)})\int_{|l|\le\frac32u}\frac{\textrm{d}^d
 l}{|l|^{\alpha\wedge2}}\le O\big(u^{d-3(\alpha\wedge2)}\big).
\end{align*}
This completes the proof of \refeq{Jbd}, as required.

Finally, we discuss the integral of $\hat I_4$.  We only need consider the
contribution from $|\Delta_{\vec u}\hat\varphi_p(l,e^{i\theta})||\Delta_{\vec
v}\hat\varphi_p(l,me^{i\theta})|$ in $\hat Y_{\vec u}(l,e^{i\theta})\hat
Y_{\vec v}(l,me^{i\theta})$ of \refeq{I4-def}, since the contributions from the
other combinations are bounded similarly to the integrals of $\hat I_1,\hat
I_2,\hat I_3$ as long as $d>2(\alpha\wedge2)$,
$\delta_1<\alpha\wedge2\wedge(d-2(\alpha\wedge2))$ and
$\delta_2<\alpha\wedge2$.  Using \refeq{Deltavarphi-bd} and ignoring some
factors of $|\theta|$, we have
\begin{align}\lbeq{I4bd1}
&\int\frac{\textrm{d}^dl}{(2\pi)^d}\int\frac{\textrm{d}\theta}{2\pi}\;
 |\Delta_{\vec u}\hat\varphi_p(l,e^{i\theta})|\,|\Delta_{\vec v}\hat\varphi_p
 (l,me^{i\theta})|\nn\\
&\le\sum_{(j_1,j'_1),(j_2,j'_2)}\int\frac{\textrm{d}^dl}{(2\pi)^d}\;\frac{O(1
 -\hat D(\vec u))}{(1-\hat D(l+j_1\vec u))(1-\hat D(l+j'_1\vec u))}\nn\\
&\hskip4pc\times\int\frac{\textrm{d}\theta}{2\pi}\;\frac{1-\hat D(\vec v)}
 {(|\theta|+1-\hat D(l+j_2\vec v))(|\theta|+1-\hat D(l+j'_2\vec v))}.
\end{align}
Notice that
\begin{align*}
&\sum_{(j_2,j'_2)}\int\frac{\textrm{d}\theta}{2\pi}\;\frac1
 {(|\theta|+1-\hat D(l+j_2\vec v))(|\theta|+1-\hat D(l+j'_2\vec v))}\\
&\le\sum_{\substack{(j_2,j'_2)}}\frac1{1-\hat D(l+j_2\vec v)\vee\hat D
 (l+j'_2\vec v)}\le\sum_{j=0,\pm1}\frac2{1-\hat D(l+j\vec v)}.
\end{align*}
The contribution from $j=0$ is bounded, similarly to \refeq{I1-bd2}, by
$O(1-\hat D(\vec u))\hat J(u)(1-\hat D(\vec v))$, where $(1-\hat D(\vec u))\hat
J(u)/u^{1+\delta_1}$ is integrable if
$\delta_1<\alpha\wedge2\wedge(d-2(\alpha\wedge2))$ and $(1-\hat D(\vec
v))/v^{1+\delta_2}$ is integrable if $\delta_2<\alpha\wedge2$ (see around
\refeq{Jbd}).  On the other hand, the contribution from $j=\pm1$ is bounded,
due to the Schwarz inequality and the $\Zd$-symmetry and translation-invariance
of $D$, by
\begin{align}\lbeq{I4bd2}
&\sum_{(j_1,j'_1)}\int\frac{\textrm{d}^dl}{(2\pi)^d}\;\frac{O(1-\hat D(\vec
 u))}{(1-\hat D(l+j_1\vec u))(1-\hat D(l+j'_1\vec u))}\frac{1-\hat D(\vec v)}
 {1-\hat D(l+j\vec v)}\nn\\
&\le\sum_{(j_1,j'_1)}\bigg(\int\frac{\textrm{d}^dl}{(2\pi)^d}\;\frac{O(1-\hat
 D(\vec u))^2}{(1-\hat D(l+j_1\vec u))^2(1-\hat D(l+j'_1\vec u))}\bigg)^{1/2}
 \nn\\
&\hskip5pc\times\bigg(\int\frac{\textrm{d}^dl}{(2\pi)^d}\;\frac{(1-\hat D(\vec
 v))^2}{(1-\hat D(l+j'_1\vec u))(1-\hat D(l+j\vec v))^2}\bigg)^{1/2}\nn\\
&\le O\big(1-\hat D(\vec u)\big)\big(1-\hat D(\vec v)\big)\sum_{(j_1,j'_1)}\hat
 J\big((1-j_1j'_1)u\big)^{1/2}\hat J(|v-jj'_1u|)^{1/2}\nn\\
&=O\big(1-\hat D(\vec u)\big)\big(1-\hat D(\vec v)\big)\Big(\big(\hat J(u)^{1
 /2}+\hat J(2u)^{1/2}\big)\,\hat J(|v+ju|)^{1/2}\nn\\
&\hskip13pc+\hat J(u)^{1/2}\hat J(|v-ju|)^{1/2}\Big).
\end{align}
It is not hard to show that $\hat J(2u)$ and $\hat J(v+u)$ obey the same bound
as $\hat J(u)$ for $u,v\in[0,1]$.  Therefore, the contribution to \refeq{I4bd2} from $\hat J(v+u)$ is bounded by $O(1-\hat D(\vec u))(1-\hat D(\vec v))\hat J(u)$, which divided by $u^{1+\delta_1}v^{1+\delta_2}$ is integrable if
$d>2(\alpha\wedge2)$, $\delta_1<\alpha\wedge2\wedge(d-2(\alpha\wedge2))$ and
$\delta_2<\alpha\wedge2$, as explained above.  Moreover, since
\begin{align*}
\int_{\frac{u}2}^{\frac{3u}2}\frac{\textrm{d}v}{v^{1+\delta_2}}\,\big(1-\hat D
 (\vec v)\big)\hat J(|v-u|)^{1/2}
&\le\underbrace{\int_0^{\frac{u}2}\textrm{d}r\;\hat J(r)^{1/2}}_{O\big(u^{\frac
 {d-3(\alpha\wedge2)}2\wedge0+1}\big)}\times
 \begin{cases}
 O\big(u^{(\alpha\wedge2-\delta_2-1)\wedge0}\big)&(\alpha\ne2)\\
 O\big(u^{(1-\delta_2)\wedge0}\log\tfrac1u\big)\quad&(\alpha=2)
 \end{cases}\\
&\le O\Big(u^{\frac{d-3(\alpha\wedge2)}2\wedge0}\Big),
\end{align*}
and
\begin{align*}
\int_{[0,1]\setminus[\frac{u}2,\frac{3u}2]}\frac{\textrm{d}v}{v^{1+\delta_2}}\,
 \big(1-\hat D(\vec v)\big)\hat J(|v-u|)^{1/2}
&\le O\Big(u^{\frac{d-3(\alpha\wedge2)}2\wedge0}\Big)\int_0^1\frac{\textrm{d}v}
 {v^{1+\delta_2}}\,\big(1-\hat D(\vec v)\big)\\
&\le O\Big(u^{\frac{d-3(\alpha\wedge2)}2\wedge0}\Big),
\end{align*}
we have
\begin{align}\lbeq{I4bd3}
\int_0^1\frac{\textrm{d}v}{v^{1+\delta_2}}\,\big(1-\hat D(\vec v)\big)\hat J
 (|v-u|)^{1/2}\le O\Big(u^{\frac{d-3(\alpha\wedge2)}2\wedge0}\Big),
\end{align}
i.e., the left-hand side of \refeq{I4bd3} obeys the same bound as $\hat
J(u)^{1/2}$.  Therefore, the contribution to \refeq{I4bd2} from $\hat J(|v-u|)$, divided by $u^{1+\delta_1}v^{1+\delta_2}$, is also integrable if $d>2(\alpha\wedge2)$, $\delta_1<\alpha\wedge2\wedge(d-2(\alpha\wedge2))$ and
$\delta_2<\alpha\wedge2$, as required.  This completes the proof of
Lemma~\ref{lmm:Ii-bds}.
\end{proof}

\section*{Acknowledgements}
The work of LCC was carried out at the University of British Columbia, Canada,
and was supported in part by TJ \& MY Foundation and NSC grant.  The work of AS
was supported by the 2007 Special Coordination Funds for Promoting Science and
Technology of the Ministry of Education, Culture, Sports, Science and
Technology, of Japan.  AS would like to thank the secretaries of ``L-Station"
at CRIS and those of the Department of Mathematics at Hokkaido University for
having set up a comfortable working environment as quickly as they could during
the transition period.  LCC and AS are grateful to the anonymous referees for
various useful suggestions.

\end{document}